\documentclass[10pt,a4paper]{article}
\usepackage[english]{babel}
\usepackage{lipsum}
\usepackage{commath}
\usepackage{graphicx}
\usepackage{enumitem}

\usepackage{float}
\usepackage{amsthm,amssymb,mathrsfs,setspace,amsmath}

\usepackage{graphicx}
\usepackage{caption}
\usepackage{subcaption}

\usepackage[a4paper,bindingoffset=0.2in,
left=0.5in,right=0.5in,
top=1in,bottom=1in,footskip=.25in]{geometry}

\title{System of Lane-Emden equations as IVPs BVPs and Four Point BVPs\\\&\\Computation with Haar Wavelets}
\author{Amit K. Verma$^a$\thanks{Corresponding author email: akverma@iitp.ac.in}, Narendra Kumar$^b$, Diksha Tiwari$^c$\\\small{\it{$^{a,b}$Department of Mathematics}} \\\small{\it{Indian Institute of Technology Patna,}}\\\small{\it{ Bihta, Patna 801103, (BR) India.}}\\\small{\textit{$^c$Faculty of Mathematics, University of Vienna, Austria.}}}
\date{\today}
\theoremstyle{definition}
\newtheorem{definition}{Definition}[section]
\newtheorem{theorem}{Theorem}[section]

\begin{document}

\maketitle

\begin{abstract}
In this work we present Haar wavelet collocation method and solve the following class of system of Lane-Emden equation defined as
\begin{eqnarray*}
-(t^{k_1} y'(t))'=t^{-\omega_1} f_1(t,y(t),z(t)),\\
-(t^{k_2} z'(t))'=t^{-\omega_2} f_2(t,y(t),z(t)),
\end{eqnarray*}
where $t>0$, subject to initial values, boundary values and four point boundary values: 
\begin{eqnarray*}
\mbox{Initial Condition:}&&y(0)=\gamma_1,~y'(0)=0,~z(0)=\gamma_2,~z'(0)=0,\\
\mbox{Boundary Condition:}&&y'(0)=0,~y(1)=\delta_1,~z'(0)=0,~z(1)=\delta_2,\\
\mbox{Four~point~Boundary~Condition:}&&y(0)=0,~y(1)=n_1z(v_1),~z(0)=0,~z(1)=n_2y(v_2),
\end{eqnarray*}
where $n_1$, $n_2$, $v_1$, $v_2$ $\in (0,1)$ and $k_1\geq 0$, $k_2\geq0$, $\omega_1<1$, $\omega_2<1$ are real constants. Results are compared with exact solutions in the case of IVP and BVP. In case of four point BVP we compare the result with other methods. Convergence of these methods is also established and found to be of second order. We observe that as resolution is increased to $J=4$ we get the exact values for IVPs and BVPs. For four point BVPs also at $J=4$, we get highly accurate solutions, e.g., the $L^\infty$ error is of order $10^{-16}$ or $10^{-17}$. 
\end{abstract}

\section{Introduction} Lane-Emden type of equations in one dependent variable occur in several branches of sciences and engineering \cite{CS1967, PL1952, RA1986, RW1989, JV1998}. Excellent treatment of existence uniqueness of such a generalised class of Lane-Emden type equations can be found in  \cite{PANDEY2008NARWA, PANDEY2009NATMA, PANDEY2008JMAA, VERMA2011NATMA, Pandey2010JAMC, PANDEY2010AMC, akviitkgp2009, Verma2019AADM} and the references cited there in. Existence of solutions for three point BVPs can be found in \cite{Singh2013IJDE, Verma2014EJDE, Verma2015JOMC1, Singh2015JOMC2, Verma2015MMA, Singh2017CAA} and for numerical solution one may refer \cite{Singh2019JAAC}. In a recent work \cite{VermaIJCM2019} existence is proved for a class of four point BVPs. Recently Haar wavelets are used efficiently to solve Lane-Emden equations for example please refer the papers \cite{Amit-Diksha-2018, RandhirSingh2019, Randhir-Himanshu-2019}. Other wavelets are also utilised to solve Lane-Emden equations \cite{Amit-Diksha-2019-1, Amit-Diksha-2019-2}.

Ma \cite{MA2000NATMA}  considered the following set of differential and boundary equations:
\begin{eqnarray}
&&u''(t)+\lambda a(t) f (u(t), v(t))=0, \quad t \in (0, 1),\\
&&v''(t)+\lambda b(t) g(u(t),v(t))=0, \quad t \in (0, 1),\\
&&u(0) = v(0)=u(1)=v(1)=0.
\end{eqnarray}
Under suitable conditions on $a(t),b(t),f,g$ and using theory of cones existence of multiple positive solutions is proved in \cite{MA2000NATMA}.

For 4-point BVP the work of Zhang et al. \cite{Zhang2009} may be referred, in which singularity was allowed at both end points of the BVP. They considered
\begin{eqnarray}
&&-u''(t) = f (t,v(t)), \quad t \in (0, 1),\\
&&-v''(t) = g(t,u(t)), \quad t \in (0, 1),\\
&&u(0) = a u(\xi), \quad u(1) = b u(c),\\
&&v(0) = a v(\xi), \quad v(1) = b v(c),
\end{eqnarray}
where $0 <\xi<c< 1$, $0 \leq a < \frac{1}{1-\xi}$ ,  $ 0 \leq b<\frac{1}{c}$ , $a \xi (1 - 
b) + (1 - a)(1 - b c) > 0$; $f (t,v)$ and $g(t,u)$ may be singular at $t = 0 $ and/or $t = 1$. Based on functional expansion-compression fixed point theorem existence of positive solution is achieved. 

Wang et al. \cite{Wang2009} by using the fixed point theory of cone expansion and compression with norm type, studied the existence of positive solutions for a class of second-order nonlinear $m-$point boundary value problems of differential system.

Hao et al. \cite{Hao2018} considered the following set of equations
\begin{eqnarray}
&&-u''(x)+\frac{r}{x}u'(x) = f_1 (x,u(x),v(x)), \quad x \in (0, 1),\\
&&-v''(x) +\frac{r}{x}v'(x) = f_2 (x,u(x),v(x)), \quad x \in (0, 1),
\end{eqnarray}
subject to conditions
\begin{eqnarray}
&&u'(0) = 0, \quad u(1) = \beta_1,\\
&&v'(0) = 0, \quad v(1) = \beta_2,
\end{eqnarray}
where $r$ is a real constant and $f_1(u(x), v(x))$ and $f_2(u(x), v(x))$ are arbitrary functions of $u$ and $v$. They used  series representation to approximate the solution of coupled Lane-Emden boundary value problems. Xie et al. \cite{Xie2019} also solved similar type of Lane-Emden problems by using differential transform method coupled with Adomian polynomials.

Yalcin \cite{Yalcin2018} proposed an efficient numerical method for solving system of Lane-Emden type equations using Chebyshev operational matrix method. He et al. \cite{He2019} used Taylor series method to solve system of Lane-Emden equations of second order. 

Madduri et al. \cite{Madduri2019JOMC} presented a fast-converging iterative scheme to approximate the solution of a system of Lane-Emden equations arising in catalytic diffusion reactions. Mahalakshmi et al. \cite{Mahalakshmi2019JOMC} presented an efficient wavelet-based method for the numerical solutions of nonlinear coupled reaction-diffusion equations which occurs in biochemical engineering. 

Very recently, Barnwal et al. \cite{Barnwal2019} conisdered the following nonlinear four point BVP,
\begin{eqnarray}
&&-(p(x)u'(x)' = q(x) f_1 (x,u(x),v(x)),\quad x \in (0, 1),\\
&&-(p(x)v'(x))' =q(x) f_2 (x,u(x),v(x)),\quad x \in (0, 1),
\end{eqnarray}
subject to conditions
\begin{eqnarray}
&&u(0) = v(0)=0,\quad u(1) = \mu_1 z(v_1), \quad u(1) = \mu_2z(v_2).
\end{eqnarray}
Here $p(0)=0$, $\int_0^\infty(1/p)dt<\infty$ and $q(x)>0$ such that $q\in L^1[0,1]$. Such BVPs are also referred as doubly singular BVPs. They used monotonic iterative technique without the presence of upper-lower solutions. 

In the present work we use Haar wavelet collocation approach together with Newton Raphson method to solve system of nonlinear Lane-Emden equations
\begin{eqnarray}
\label{p2eqn1}-(t^{k_1} y'(t))'=t^{-\omega_1} f_1(t,y(t),z(t)),\\
\label{p2eqn2}-(t^{k_2} z'(t))'=t^{-\omega_2} f_2(t,y(t),z(t)),
\end{eqnarray}
where $t>0$, subject to initial values, boundary values and four point boundary values: 
\begin{eqnarray*}
\mbox{Initial Condition:}&&y(0)=\gamma_1,~y'(0)=0,~z(0)=\gamma_2,~z'(0)=0,\\
\mbox{Boundary Condition:}&&y'(0)=0,~y(1)=\delta_1,~z'(0)=0,~z(1)=\delta_2,\\
\mbox{Four~point~Boundary~Condition:}&&y(0)=0,~y(1)=n_1z(v_1),~z(0)=0,~z(1)=n_2y(v_2),
\end{eqnarray*}
where $n_1$, $n_2$, $v_1$, $v_2$ $\in (0,1)$ and $k_1\geq 0$, $k_2\geq0$, $\omega_1<1$, $\omega_2<1$ are real constants. We developed three methods referred as Haar wavelet collocation approach for IVPs (HWCAIVP), Haar wavelet collocation approach for BVPs (HWCABVP) and Haar wavelet collocation approach for four point BVPs (HWCA4PTBVP). For each case we have considered two test examples. Convergence of the proposed methods are also proved. Results obtained in this paper are new and does not exist in literature. Calculations and plotting are performed on Mathematica 11.3 and Origin 8.

The paper is arranged in several sections. In section \ref{SecHaar} we  introduce about wavelets and Haar wavelets. Section \ref{Sec2PIVP} is devoted to nonlinear singular coupled system of IVPs. Section \ref{Sec2PBVP}  is devoted to nonlinear singular coupled system of IVPs. Section \ref{Sec4pBVP} is devoted to nonlinear singular coupled system of four point BVPs. Lastly in section \ref{remarkfinal} we conclude the paper with some important remarks.
\section{Preliminary}\label{SecHaar}
Wavelet word is used in the name of French geophysicist Jean Morlet (1931--2007). Morlet and Croatian-French physicist Alex Grossman developed the modern wavelet theory  \cite[p. 222]{MCPLAW2012}. Wavelets are multi-indexed and it has parameters which is used to shift or dilate/contract the functions giving us basis functions. The following properties of wavelets enable us to choose them over other methods \cite{MCPLAW2012, BN2009, MALLAT1989,GL2000}.
\begin{description}
\item[(H1)] Orthogonality.
\item[(H2)] Compact Support.
\item[(H3)] Density.
\item[(H4)] Localization Property: Haar basis is localized, i.e., the vector is zero except for a few entries.
\item[(H5)] Multiresolution Analysis (MRA).
\end{description}

\subsection*{MRA}
An orthogonal multiresolution analysis (MRA) is a collection of closed subspaces of $L^2(\mathbb{R})$ which are nested, having trivial intersection, they exhaust the space, the subspaces are connected to each other by scaling property and finally there is a special function, the scaling function $\varphi$, whose integer translates form an orthonormal basis for
one of the subspaces. We give formal statement of MRA as defined in \cite{MCPLAW2012}.
\begin{definition}
An MRA with scaling function $\varphi$ is a collection of closed subspaces $V_j , j = \dots,-2,-1,0,1,2,\dots$ of $L^2(\mathbb{R})$ such that
\begin{enumerate}
\item $\mathbf{V}_{j} \subset \mathbf{V}_{j+1}$.
\item $\overline{\bigcup\mathbf{V}_{j}}=\mathbf{L^2(\mathbb{R})}$.
\item $ \bigcap \mathbf{V}_{j}={0}$.
\item The function $f(x)$ belongs to $\mathbf{V}_{j}$ if and only if the function $f(2x) \in \mathbf{V}_{j+1}$.
\item The function $\varphi$ belongs to $\mathbf{V}_{0}$, the set $\{\varphi(x-k), k \in \mathbf{\mathbb{Z}}\}$ is orthonormal basis for $\mathbf{V}_{0}$.
 \end{enumerate}
\end{definition}
The sequence of wavelet subspaces $W_j$ of $L^2(\mathbb{R})$, are such that $V_j\perp Wj$ , for all $j$ and $V_{j+1} = V_j\oplus W_j$. Closure of $\oplus_{j\in \mathbb{Z}} W_j$ is dense in $L^2(\mathbb{R})$ with respect to $L^2$ norm.

Mallat's theorem \cite{MALLAT1989} guarantees that in presence of an orthogonal MRA, an orthonormal basis for $L^2(\mathbb{R})$ exists. 

\begin{theorem} {$\mathrm{(Mallat's~Theorem)}$}. Given an orthogonal MRA
with scaling function $\varphi$, there is a wavelet $\psi\in L^2(\mathbb{R})$ such that for
each $j\in \mathbb{Z}$, the family $\{\psi_{j,k}\}_{k\in\mathbb{Z}}$ is an orthonormal basis for $W_j$ .
Hence the family $\{\psi_{j,k}\}_{k\in\mathbb{Z}}$ is an orthonormal basis for $L^2(\mathbb{R})$. \label{Mallat}
\end{theorem}
\subsection{Haar Matrices} Excellent treatment is given on applications of Haar Wavelets in \cite{HaarLepik2014} by Lepik.  Let us consider the interval $0\leq x \leq 1$. Let us define $M=2^J$,where $J$ is maximum level of resolution. Let us divide $[0,1]$ into $2M$ subintervals of equal length $\Delta x=\frac{1}{2M}$. The wavelet number $i$ is calculated by $i=m+k+1$, here $j=0,1,\cdots,J$ and $k=0,1,\cdots,m-1$ (here $m=2^j$). The Haar wavelet's mother wavelet function is defined as \cite[pp.7-10]{HaarLepik2014} ,
\begin{equation} \label{p2eq1}
h_i(x)=
\begin{cases} 
     1,  & {\eta_1(i)\leq x< \eta_2(i)}, \\
    -1,  & {\eta_2(i)\leq x< \eta_3(i)}, \\
     0,  & \quad \text{else},
   \end{cases}
\end{equation}
where,
\begin{equation} \label{p2eq2} \quad \eta_1(i)=2k \mu \Delta x , \quad \eta_2(i)=(2k+1) \mu \Delta x , \quad \eta_3(i)=2(k+1) \mu \Delta x , \quad \mu =\frac{M}{m},
\end{equation} \newline
if $i>2$. For $i=1$, it is defined as,
\begin{equation} \label{p2eq3}
h_1(x)=
\begin{cases} 
     1,  & {0 \leq x \leq 1}, \\
      0, & \quad \text{else}.
   \end{cases}
\end{equation}
For $i=2$, 
\begin{equation} \label{p2eq4}
\quad \eta_1(2)=0 , \quad \eta_2(2)=0.5 , \quad \eta_3(2)=1.
\end{equation}
The width of the $i^{th}$ wavelet is,
\begin{center}
{$\eta_3(i)-\eta_1(i)= 2 \mu \Delta x = 2^{-j}$}.
\end{center}
The integral $P_{v,i}(x)$ is defined as,
\begin{equation} \label{p2eq5}
P_{v,i}(x)=\int^x_0 \int^x_0 \cdots \int^x_0 h_i(t) dt^v = \frac{1}{(v-1)!} \int^x_0 (x-t)^{v-1} h_i(t) dt,
\end{equation}
here $v$ is the order of integration. \newline
Using \eqref{p2eq1}, we will calculate these integrals analytically and by doing it we obtain,
\begin{equation} \label{p2eq6}
P_{v,i}(x)=
\begin{cases} 
     0,  & {x< \eta_1(i)}, \\
    \frac{1}{v!}[x-\eta_1(i)]^v,  & {\eta_1(i)\leq x \leq \eta_2(i)}, \\
    \frac{1}{v!}{[x-\eta_1(i)]^v - 2[x-\eta_2(i)]^v},  & {\eta_2(i)\leq x \leq \eta_3(i)},\\
      \frac{1}{v!}{[x-\eta_1(i)]^v - 2[x-\eta_2(i)]^v + [x-\eta_3(i)]^v},  & { x > \eta_3(i)},
   \end{cases}
\end{equation} 
for $i>1$. For $i=1$ we have $\eta_1=0,\eta_2=\eta_3=1$ and,
\begin{equation} \label{p2eq7}
P_{v,1}(x)=\frac{x^v}{v!}.
\end{equation}
For computation by using Haar wavelet, we make use of the method of collocation. Here, the grid points are defined by,
\begin{equation} \label{p2eq8}
\tilde x_c=c \Delta x,\quad c=0,1,\cdots,2M,
\end{equation}
and collocation points are defined by,
\begin{equation} \label{p2eq9}
x_c=0.5(\tilde x_{c-1}+\tilde{x_c}),\quad c=1,\cdots,2M,
\end{equation}
and replace $x \rightarrow x_c$ in \eqref{p2eq1}, \eqref{p2eq6}, \eqref{p2eq7}.

For computational point of view, we introduce the Haar matrices $H,P_1,P_2,\cdots ,P_v$. The order of these matrices are $2M \times 2M$. These matrices are defined by,
\begin{center}
$H(i,c)=h_i(x_c),\quad P_v(i,c)=p_{v,i}(x_c) ~~~v=1,2,\cdots.$
\end{center} 
For $J=1$, the matrices $H$, $P_1$  and $P_2$ are defined by,
\begin{center}
$H=
\begin{bmatrix}
1&1&1&1 \\ 1&1&-1&-1 \\ 1&-1&0&0 \\ 0&0&1&-1
\end{bmatrix}$,~~ $P_1= \frac{1}{8} \begin{bmatrix}
1&3&5&7 \\ 1&3&3&1 \\ 1&1&0&0 \\ 0&0&1&1
\end{bmatrix}$,~~ $P_2= \frac{1}{128} \begin{bmatrix}
1&9&25&49 \\ 1&9&23&31 \\ 1&7&8&8 \\ 0&0&1&7
\end{bmatrix}.$ 
\end{center}

\section{System of Nonlinear Singular Two Point IVP}\label{Sec2PIVP}

In this section we develop the solution method for system of nonlinear singular two point initial value problems. We will study some examples based on it.

\subsection{Method 1: HWCAIVP}\label{method2PIVP}

Consider the following singular nonlinear system of differential equations \eqref{p2eqn1}-\eqref{p2eqn2} subject to the initial conditions, 
\begin{equation}\label{p2eq12}
y(0)=\gamma_1,~y'(0)=0,
\end{equation} 
\begin{equation}\label{p2eq13}
z(0)=\gamma_2,~z'(0)=0,
\end{equation}
where $t>0$ and where $k_1\geq 0$, $k_2\geq0$, $\omega_1<1$, $\omega_2<1$, $\gamma_1,\gamma_2$ are real constants. .

\begin{theorem}
Consider the differential equations \eqref{p2eqn1}-\eqref{p2eqn2} with initial conditions \eqref{p2eq12}-\eqref{p2eq13}. Let us assume $f_1(t,y,z)$, $f_2(t,y,z)$ be continuous functions in $t$. Let $y(t)$ and $z(t)$ be the solutions of the differential equations \eqref{p2eqn1}-\eqref{p2eqn2} subject to the initial conditions \eqref{p2eq12}-\eqref{p2eq13}. Then the numerical solutions $y(t)$ and $z(t)$ for differential equations \eqref{p2eqn1}-\eqref{p2eqn2} using Haar wavelet method are given as follows,
\begin{equation*}
y(t)=\gamma_1+\sum_{i=1}^{2M} a_i P_{2,i}(t),
\end{equation*}
\begin{equation*}
z(t)=\gamma_2+\sum_{i=1}^{2M} b_i P_{2,i}(t).
\end{equation*}
\end{theorem}

\begin{proof}
Assume (\cite{HaarLepik2008}),
\begin{equation}\label{p2eq14}
y''(t)=\sum_{i=1}^{2M} a_i h_i(t),
\end{equation}
\begin{equation}\label{p2eq15}
z''(t)=\sum_{i=1}^{2M} b_i h_i(t),
\end{equation}
where $a_i$, $b_i$ are wavelet coefficients. Now integrate above equations from $0$ to $t$ twice we get,
\begin{equation}\label{p2eq16}
y'(t)=\sum_{i=1}^{2M} a_i P_{1,i}(t) + y'(0),
\end{equation}
\begin{equation}\label{p2eq17}
y(t)=\sum_{i=1}^{2M} a_i P_{2,i}(t) + ty'(0)+ y(0).
\end{equation}
\begin{equation}\label{p2eq18}
z'(t)=\sum_{i=1}^{2M} b_i P_{1,i}(t) + z'(0),
\end{equation}
\begin{equation}\label{p2eq19}
z(t)=\sum_{i=1}^{2M} b_i P_{2,i}(t) + tz'(0)+ z(0).
\end{equation}
Now apply initial conditions \eqref{p2eq12}-\eqref{p2eq13} in \eqref{p2eq16}-\eqref{p2eq19} and we get,
\begin{equation}\label{p2eq20}
y'(t)=\sum_{i=1}^{2M} a_i P_{1,i}(t),
\end{equation} 
\begin{equation}\label{p2eq21}
y(t)=\gamma_1+\sum_{i=1}^{2M} a_i P_{2,i}(t),
\end{equation} 
\begin{equation}\label{p2eq22}
z'(t)=\sum_{i=1}^{2M} b_i P_{1,i}(t),
\end{equation} 
\begin{equation}\label{p2eq23}
z(t)=\gamma_2+\sum_{i=1}^{2M} b_i P_{2,i}(t).
\end{equation} 
\end{proof}

Now substitute these equations \eqref{p2eq14}-\eqref{p2eq23} in \eqref{p2eqn1}-\eqref{p2eqn2} after discretizing by collocation method, we get the system of nonlinear equations given as,
\begin{equation} \label{p2eq24}
\Phi^{IVP}_c (a_1,a_2,\cdots,a_{2M})=0, ~~~c=1,2,\cdots,2M,
\end{equation} 
\begin{equation} \label{p2eq25}
\Psi^{IVP}_c (b_1,b_2,\cdots,b_{2M})=0, ~~~c=1,2,\cdots,2M.
\end{equation}
The above system of nonlinear equations is solved by applying Newton-Raphson method to get the wavelet coefficients $a_i$ and $b_i$ for numerical solution. Finally substituting the values of $a_i,b_i$ in equations \eqref{p2eq21} and \eqref{p2eq23}, we get the approximate solution obtained by HWCAIVP.
\subsection{Convergence Analysis for HWCAIVP}
Consider the following singular nonlinear system of differential equations, \eqref{p2eqn1}-\eqref{p2eqn2} subject to the initial conditions \eqref{p2eq12}-\eqref{p2eq13}. Here we define $M=2^J$, where is $J$ is maximum level of resolution. Let us write the approximate solutions as follows so that we have a clear distinction between exact and computed solution.
\begin{eqnarray*}
	&&y^{M} = \gamma_{1}+\sum_{0}^{2M}a_{i}p_{2,i}(t),\\
	&&|E_{M1}|=\left|\sum_{j=J+1}^{\infty}\sum_{k=0}^{2^j-1}a_{2^j+k+1}p_{2,2^j+k+1}(t)\right|,\\
	&&z^{M} =\gamma_{2}+\sum_{0}^{2M}a_{i}(p_{2,i}(t)),\\
	&&|E_{M2}|=\left|\sum_{j=J+1}^{\infty}\sum_{k=0}^{2^j-1}b_{2^j+k+1}p_{2,2^j+k+1}(t)\right|,
\end{eqnarray*}
where $|E_{M1}|=|y^{M}-y|$ and $|E_{M2}|=|z^{M}-z|$. Now, let us define 
\begin{eqnarray}\label{con1}
||E_M||_2=||E_{M1}||_2+||E_{M2}||_2.
\end{eqnarray}
\begin{theorem}\label{theoremivp}
	Consider the differential equation \eqref{p2eqn1}-\eqref{p2eqn2}  subject to the initial conditions \eqref{p2eq12}-\eqref{p2eq13} . Let us assume that  $y'''(t),z'''(t)\in L^2(\mathbb{R})$ are continuous functions on $[0,1]$. Consider $y'''_{r+1}(t)$, $z'''_{r+1}(t)$  are bounded such that $\left|y'''_{r+1}(t)\right|\leq\xi_1$ and $\left|z'''_{r+1}(t)\right|\leq\xi_2$\quad $\forall t \in[0,1]$. Let $\epsilon>0$ be arbitrary small positive number, and if $J>\log_2\left(\sqrt{\frac{\xi_1}{3\epsilon}}\right)-1$, then $||E_M||_2<\epsilon$. 
\end{theorem}
\begin{proof}
	We will do calculation for $E_{M1}$, for $E_{M2}$ it will follow accordingly.
	Expanding quadrate of $L^2$ norm of error function, we obtain
	\begin{eqnarray*}
		||E_{M1}||_{2}^2= \int_{0}^{1}{\left(\sum_{j=J+1}^{\infty}\sum_{k=0}^{2^j-1}a_{2^j+k+1}p_{2,2^j+k+1}(t)\right)}^2 dt.
	\end{eqnarray*}
	Then we have
	\begin{eqnarray}\label{bz1}
	||E_{M1}||_{2}^2 = \sum_{j=J+1}^{\infty}\sum_{k=0}^{2^j-1}\sum_{r=J+1}^{\infty}\sum_{s=0}^{2^r-1}a_{2^j+k+1}a_{2^r+s+1}\int_{0}^{1}(p_{2,2^j+k+1}(t)(p_{2,2^r+s+1}(t))dt.
	\end{eqnarray}
	By the use of definition of Haar wavelet function we can evaluate coefficients $a_{i}$ as
	\begin{eqnarray*}
		&&a_{i}=2^j\int_{0}^{1}y''(t)h_{i}(t)dt,\\
		&&a_{i}=2^j\left[\int_{\eta_{1}}^{\eta_{2}}y''(\zeta)d\zeta-\int_{\eta_{2}}^{\eta_{3}}y''(\zeta)d\zeta\right],\\
		&&a_{i}=2^j[(\eta_{2}-\eta_{1})y''(\zeta_{1})-(\eta_{3}-\eta_{2})y''(\zeta_{2})],\\
	\end{eqnarray*}
	where $\zeta_{1} \in (\eta_{1},\eta_{2})$ and $\zeta_{2} \in (\eta_{2},\eta_{3})$. It follows from Eq. \eqref{p2eq2} that $(\eta_{2}-\eta_{1})=(\eta_{3}-\eta_{2})=1/(2m)=1/(2^{j+1})$, then the above expression of $a_{i}$ reduces to
	\begin{eqnarray*}
		&&a_{i}=\frac{1}{2}[y''(\zeta_{1})-y''(\zeta_{2})]=\frac{1}{2}(\zeta_{1}-\zeta_{2})\frac{dy''}{dt}(\zeta),\:\:\:  \zeta \in (\zeta_{1},\zeta_{2}).
	\end{eqnarray*}
	Let us consider that $\left(\frac{dy''}{dt}\right)$ is bounded such that $\left|\frac{dy''}{dt}\right|\leq\xi_1$, so we have
	\begin{equation}\label{bz2}
	a_{i} \leq \xi_1\left(\frac{1}{2^{j+1}}\right).
	\end{equation}
	Here we will solve for upper bound of a function $p_{2,i}$ in all subintervals. Since $p_{2,i}(t)=0$ for $t \in [0,\eta_{1}(i)]$. The function $p_{2,i}(t)$ increases monotonically in the interval $t\in[\eta_{1}(i),\eta_{2}(i)]$. Thus $p_{2,i}(t)$ achieves its upper bound at $t= \eta_{2}(i)$ as follows
	\begin{eqnarray*}
		p_{2,i}=p_{2,2^j+k+1} \leq \frac{{[\eta_{2}(i)-\eta_{1}(i)]}^2}{2}=\frac{1}{2}\left(\frac{1}{2^{j+1}}\right)^2,\:\:\: t\in[\eta_{1}(i),\eta_{2}(i)].
	\end{eqnarray*}
	\noindent In the interval $t \in [\eta_{2}(i),\eta_{3}(i)]$ the function $p_{2,i}$ is monotonically increasing if
	\begin{eqnarray}\label{this1}
	t \leq \eta_{3},
	\end{eqnarray}
	which is obviously true.
	
	This inequality \eqref{this1} can be derived from formulas \eqref{p2eq2} and \eqref{p2eq6} and condition $\frac{dp_{2,i}(t)}{dt}>0$.
	Hence maximum value of $p_{2,i}(t)$ can be obtained by substituting $t=\eta_{3}(i)$ in eq. \eqref{p2eq6} as
	\begin{eqnarray*}
		p_{2,i}(t)=p_{2,2^j+k+1}\leq\left(\frac{1}{2^{j+1}}\right)^2,\:\:\:t\in[\eta_{2}(i),\eta_{3}(i)].
	\end{eqnarray*}
	When $t \in [\eta_{3},1]$ the function $p_{2,i}(t)$ can be expanded as (by eq. \eqref{p2eq6} (see \cite{VF2001})
	\begin{eqnarray*}
		p_{2,i}(t)=\left(\frac{1}{2^{j+1}}\right)^2.
	\end{eqnarray*}
	The function $p_{2,i}(t)$ increases monotonically in $[0,1]$, since it increases monotonically in every sub interval of $[0,1]$. So upper bound of $p_{2,i}(t)$ in $[0,1]$ is given by
	\begin{eqnarray}\label{d1}
	p_{2,i}(t)\leq \left(\frac{1}{2^{j+1}}\right)^2 \:\:\: \forall t \in[0,1] .
	\end{eqnarray}
	Now inserting eq. \eqref{bz2} in eq. \eqref{bz1} we get
	\begin{eqnarray}\label{f1}
	||E_{M1}||_{2}^2 \leq {\xi_1}^2 \sum_{j=J+1}^{\infty}\sum_{k=0}^{2^j-1}\sum_{r=J+1}^{\infty}\sum_{s=0}^{2^r-1}\frac{1}{2^{r+j+2}}\int_{0}^{1}\left(p_{2,2^j+k+1}(t)\right)\left(p_{2,2^r+s+1}(t)\right)dt.
	\end{eqnarray}
	Here we have
	\begin{eqnarray}\label{h1}
	p_{2,2^j+k+1}(t)\leq\left(\frac{1}{2^{j+1}}\right)^2, \quad p_{2,2^r+s+1}(t) \leq \left (\frac{1}{2^{r+1}}\right)^2\:\:\forall t \in [0,1].
	\end{eqnarray}
	Inserting eq. \eqref{h1} in eq. \eqref{f1} we get
	\begin{eqnarray*}
		&&||E_{M1}||_{2}^2 \leq {\xi_1}^2 \sum_{j=J+1}^{\infty}\sum_{r=J+1}^{\infty}\left(\frac{1}{2^{j+1}}\right)^3\left(\frac{1}{2^{r+1}}\right)^3 2^j2^r(1-\eta_{1}),
		\\&&||E_{M1}||_{2}^2 \leq \frac{1}{36}\xi_1^2\left(\frac{1}{2^{J+1}}\right)^4.
	\end{eqnarray*}
	Hence
	\begin{eqnarray*}
		&&||E_{M1}||_{2} \leq \frac{1}{6}\xi_1\left(\frac{1}{2^{J+1}}\right)^2.
	\end{eqnarray*}
	Let $\epsilon>0$ be arbitrary small positive number and if we choose $$J>\log_2\sqrt{\frac{\xi_1}{3\epsilon}}-1,$$ then $$||E_{M1}||_{2} \leq \frac{\epsilon}{2}.$$
	Similarly following the same process for $z^M$, we can prove that $$||E_{M2}|<\frac{\epsilon}{2}.$$ 
	Hence, $$ ||E_{M}||_2<\epsilon$$ whenever resolution $J>\log_2\sqrt{\frac{\xi_1}{3\epsilon}}-1$.
\end{proof}

\subsection{Numerical Illustration for HWCAIVP}

In this section we will discuss two numerical problems based on system of nonlinear singular two point initial value problem considered by Wazwaz et al. \cite{Wazwaz2013ADM}.

\subsubsection{Example 1 (\cite{Wazwaz2013ADM})} \label{Problem1}

Consider \eqref{p2eqn1}-\eqref{p2eqn2} subject to initial conditions \eqref{p2eq12}-\eqref{p2eq13} with $\gamma_1=1$, $\gamma_2=1$, $k_1=3$, $k_2=2$, $\omega_1=-3$, $\omega_2=-2$, $f_1(t,y(t),z(t))=-4(y+z)$ and  $f_2(t,y(t),z(t))=3(y+z)$. Now applying solution method (subsection  \ref{method2PIVP}) to solve this problem, we get system of non-linear equations. Thus we arrive at \eqref{p2eq24}-\eqref{p2eq25}.  To solve the non-linear equations we use Newton Raphson method to calculate wavelet coefficients ($a_i$) and ($b_i$). After calculating wavelet coefficients ($a_i$) and ($b_i$) we get our required solution from equations \eqref{p2eq21} and \eqref{p2eq23}. The exact solution $\tilde{y}(t)$ and $\tilde{z}(t)$ of this problem are $1+t^2$ and $1-t^2$, respectively.

Here we are computing absolute error. Absolute error is defined as the difference of exact value and approximated value. Maximum error is defined as,
\begin{equation*}
L^\infty =\max_{i}\abs{\tilde{y}(t_i)-y(t_i)},
\end{equation*}
where $\tilde{y}(t)$ is the exact solution and $y(t_i)$ is HWCAIVP solution.

For initial guess $[1,1,\cdots,1]$ computed solutions for $y^M(t)$ and $z^M(t)$ is given in Table \ref{p2Table1} and Table \ref{p2Table2}, respectively for $J=3$ and $J=4$. Graph for $y^M(t)$ and $z^M(t)$ with exact solution is given in figure \ref{p2fig1} and figure \ref{p2fig2} respectively for $J=4$. Graph of absolute errors in computation of $y^M(t)$ and $z^M(t)$ is given in figure \ref{p2fig3} and figure \ref{p2fig4} respectively for $J=1$, $J=2$, $J=3$ and $J=4$.

For slight modification in initial guesses, say, $[1.01,1.01,\cdots,1.01]$ and $[1.1,1.1,\cdots,1.1]$ final solution is not affected at all, which shows the stability of the method.

\begin{table}[H]											 
\centering											
\begin{center}											
\resizebox{7cm}{2.5cm}{											
\begin{tabular}	{|c | c|  c|  c|}		
\hline
$t$	&	$J=3$	&	$J=4$	&	Exact Solution	\\\hline
0	&	1	&	1	&	1	\\\hline
0.1	&	1.01	&	1.01	&	1.01	\\\hline
0.2	&	1.04	&	1.04	&	1.04	\\\hline
0.3	&	1.09	&	1.09	&	1.09	\\\hline
0.4	&	1.16	&	1.16	&	1.16	\\\hline
0.5	&	1.25	&	1.25	&	1.25	\\\hline
0.6	&	1.36	&	1.36	&	1.36	\\\hline
0.7	&	1.49	&	1.49	&	1.49	\\\hline
0.8	&	1.64	&	1.64	&	1.64	\\\hline
0.9	&	1.81	&	1.81	&	1.81	\\\hline
1	&	2	&	2	&	2	\\\hline
$L^\infty$	&	0	&	0	&		\\\hline

\end{tabular}}								
\end{center}
\caption{\small{Solutions $y^M(t)$ for $J=3$ and $J=4$ for example \ref{Problem1}. }}	
\label{p2Table1}											
\end{table}

\begin{table}[H]											 
\centering											
\begin{center}											
\resizebox{7cm}{2.5cm}{											
\begin{tabular}	{|c | c|  c|  c|}		
\hline
$t$	&	$J=3$	&	$J=4$	&	Exact Solution	\\\hline
0	&	1	&	1	&	1	\\\hline
0.1	&	0.99	&	0.99	&	0.99	\\\hline
0.2	&	0.96	&	0.96	&	0.96	\\\hline
0.3	&	0.91	&	0.91	&	0.91	\\\hline
0.4	&	0.84	&	0.84	&	0.84	\\\hline
0.5	&	0.75	&	0.75	&	0.75	\\\hline
0.6	&	0.64	&	0.64	&	0.64	\\\hline
0.7	&	0.51	&	0.51	&	0.51	\\\hline
0.8	&	0.36	&	0.36	&	0.36	\\\hline
0.9	&	0.19	&	0.19	&	0.19	\\\hline
1	&	0	&	0	&	0	\\\hline
$L^\infty$	&	0	&	0	&		\\\hline

\end{tabular}}								
\end{center}
\caption{\small{Solutions $z^M(t)$ for $J=3$ and $J=4$ for example \ref{Problem1}.}}	
\label{p2Table2}											
\end{table}

\begin{figure}\centering
        \begin{subfigure}[b]{0.45\textwidth}
             \includegraphics[scale=0.30]{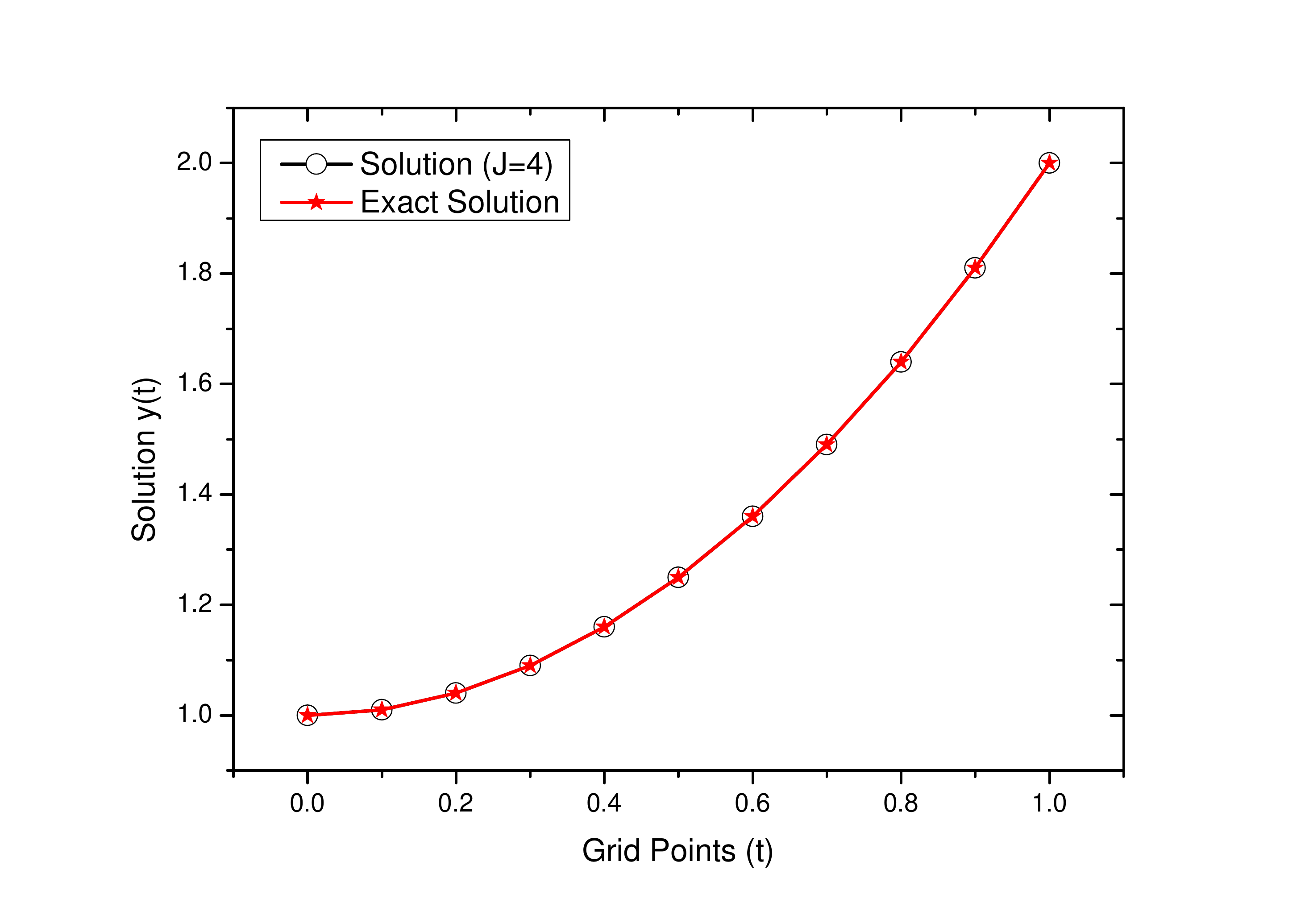}
\subcaption{Graph of $y^M(t)$ for $J=4$  for example \ref{Problem1}.}\label{p2fig1} 
        \end{subfigure}\hspace{0.5in}
        \begin{subfigure}[b]{0.45\textwidth}
\includegraphics[scale=0.30]{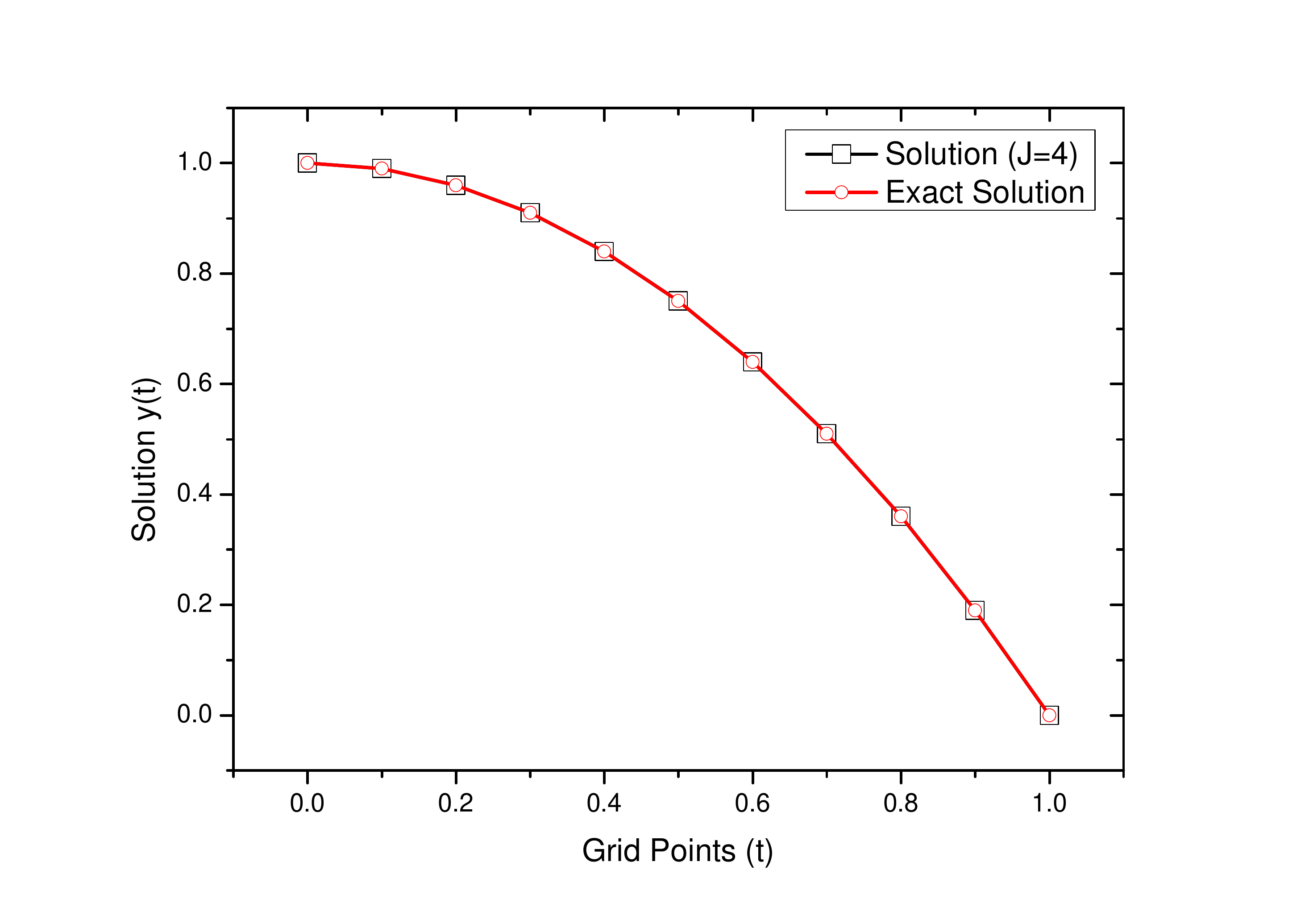}
\subcaption{Graph of $z^M(t)$ for $J=4$  for example \ref{Problem1}.}\label{p2fig2}
        \end{subfigure}%
\caption{}\label{ivpfig1}
\end{figure}

\begin{figure}\centering
        \begin{subfigure}[b]{0.45\textwidth}
\includegraphics[scale=0.3]{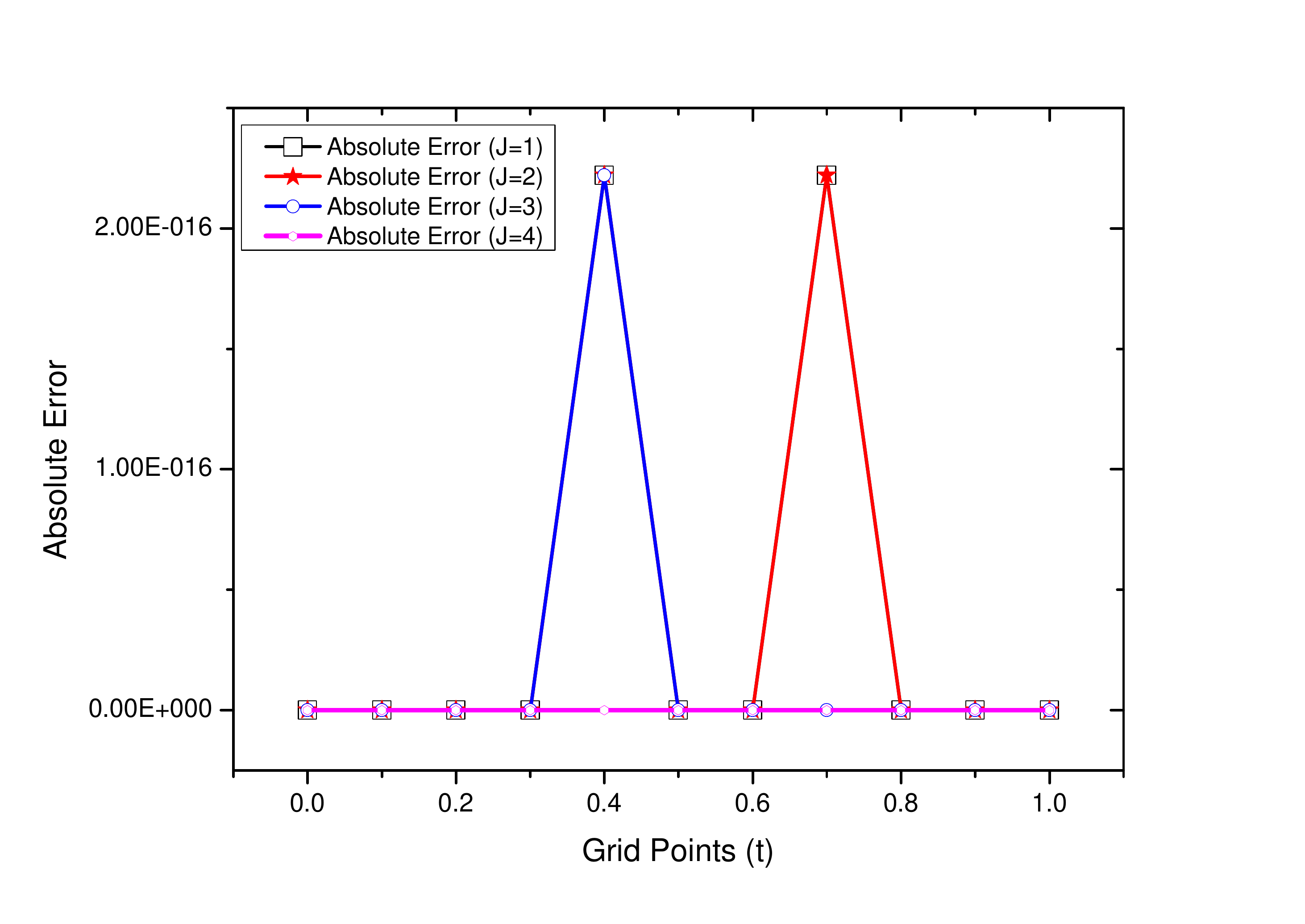}
\subcaption{Graph of absolute errors in $y^M(t)$  for example \ref{Problem1}.}\label{p2fig3}
        \end{subfigure}\hspace{0.5in}
        \begin{subfigure}[b]{0.45\textwidth}
\includegraphics[scale=0.3]{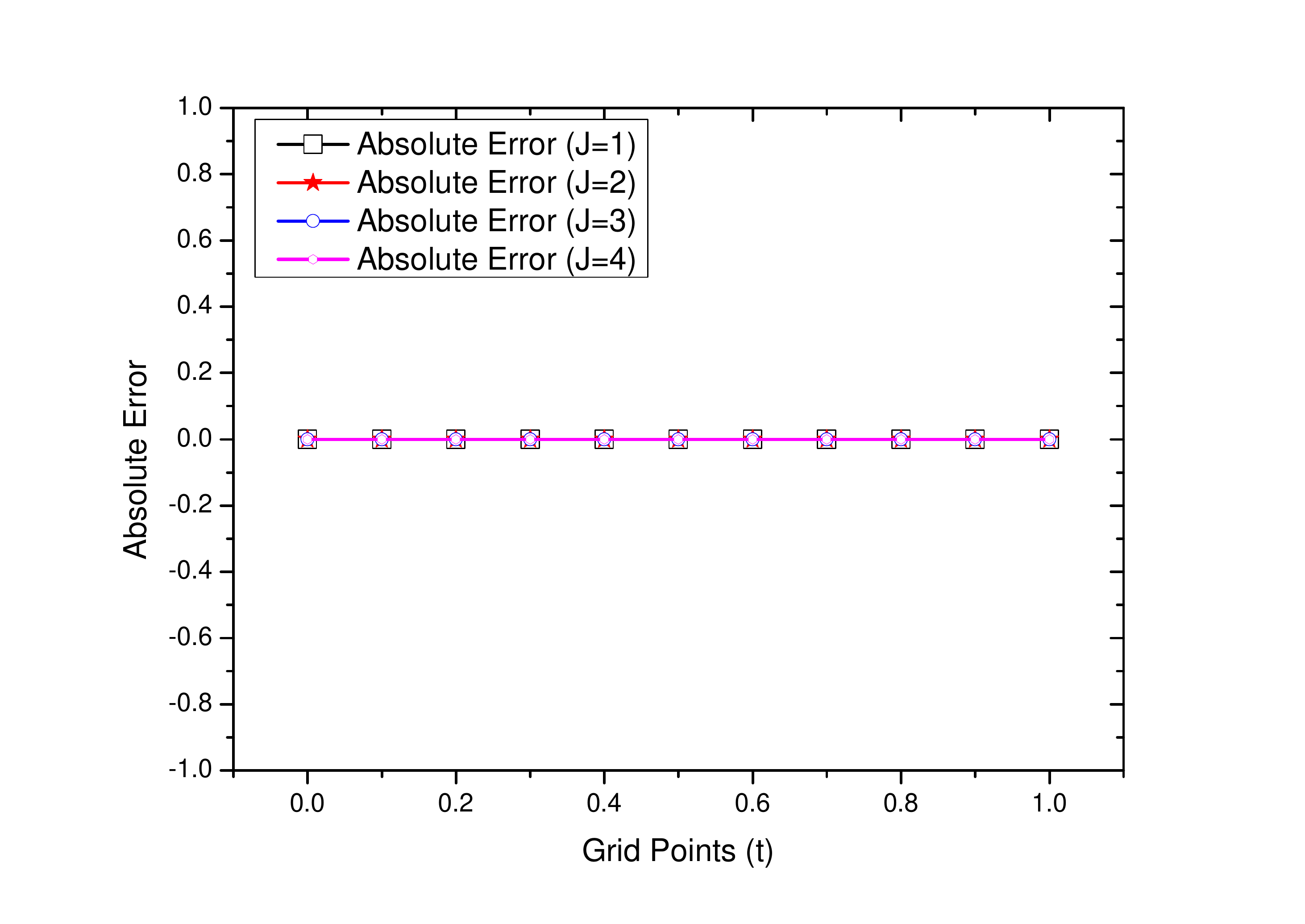}
\subcaption{Graph of absolute errors in $z^M(t)$  for example \ref{Problem1}.}\label{p2fig4}
        \end{subfigure}%
\caption{}\label{ivpfig2}
\end{figure}

\subsubsection{Example 2 (\cite{Wazwaz2013ADM})} \label{Problem2}

Consider the system of differential equations \eqref{p2eqn1}-\eqref{p2eqn2} subject to initial conditions \eqref{p2eq12}-\eqref{p2eq13}  with $\gamma_1=1$, $\gamma_2=1$, $k_1=1$, $k_2=3$, $\omega_1=-1$, $\omega_2=-3$, $f_1(t,y(t),z(t))=-z^3 (y^2+1)$ and  $f_2(t,y(t),z(t))=z^5 (y^2+3)$. The exact solution $\tilde{y}(t)$ and $\tilde{z}(t)$ of this problem are $\sqrt{1+t^2}$ and $\frac{1}{\sqrt{1+t^2}}$ respectively.

Since this problem is same as subsubsection $\ref{Problem1}$. Computed solutions for $y^M(t)$ and $z^M(t)$ is given in Table \ref{p2Table3} and Table \ref{p2Table4}, respectively for $J=3$ and $J=4$ taking initial guess $[0,0,\cdots,0]$. Graph for $y^M(t)$ and $z^M(t)$ with exact solution is given in figure \ref{p2fig5} and figure \ref{p2fig6} respectively for $J=4$. Graph of absolute errors in computation of $y^M(t)$ and $z^M(t)$ is given in figure \ref{p2fig7} and figure \ref{p2fig8} respectively for $J=1$, $J=2$, $J=3$ and $J=4$.

For slight modification in initial guesses, say, $[0.01,0.01,\cdots,0.01]$ and $[0.1,0.1,\cdots,0.1]$ final solution is not affected at all, which shows the stability of the method.

\begin{table}[H]											 
\centering											
\begin{center}											
\resizebox{7cm}{2.5cm}{											
\begin{tabular}	{|c | c|  c|  c|}		
\hline
$t$	&	$J=3$	&	$J=4$	&	Exact Solution	\\\hline
0	&	1	&	1	&	1	\\\hline
0.1	&	1.00499	&	1.00499	&	1.00499	\\\hline
0.2	&	1.0198	&	1.0198	&	1.0198	\\\hline
0.3	&	1.04403	&	1.04403	&	1.04403	\\\hline
0.4	&	1.07703	&	1.07703	&	1.07703	\\\hline
0.5	&	1.11803	&	1.11803	&	1.11803	\\\hline
0.6	&	1.16619	&	1.16619	&	1.16619	\\\hline
0.7	&	1.22065	&	1.22065	&	1.22066	\\\hline
0.8	&	1.28061	&	1.28062	&	1.28062	\\\hline
0.9	&	1.34535	&	1.34536	&	1.34536	\\\hline
1	&	1.41419	&	1.41421	&	1.41421	\\\hline
$L^\infty$	&	2.76016E-05	&	6.87997E-06	&		\\\hline

\end{tabular}}								
\end{center}
\caption{\small{Solution $y^M(t)$ for $J=3$ and $J=4$ for example \ref{Problem2}.}}	
\label{p2Table3}											
\end{table}

\begin{table}[H]											 
\centering											
\begin{center}											
\resizebox{7cm}{2.5cm}{											
\begin{tabular}	{|c | c|  c|  c|}		
\hline
$t$	&	$J=3$	&	$J=4$	&	Exact Solution	\\\hline
0	&	1	&	1	&	1	\\\hline
0.1	&	0.995035	&	0.995036	&	0.995037	\\\hline
0.2	&	0.980564	&	0.980577	&	0.980581	\\\hline
0.3	&	0.9578	&	0.957819	&	0.957826	\\\hline
0.4	&	0.928424	&	0.928465	&	0.928477	\\\hline
0.5	&	0.894355	&	0.894409	&	0.894427	\\\hline
0.6	&	0.857402	&	0.85747	&	0.857493	\\\hline
0.7	&	0.819121	&	0.819205	&	0.819232	\\\hline
0.8	&	0.780748	&	0.780839	&	0.780869	\\\hline
0.9	&	0.743164	&	0.743262	&	0.743294	\\\hline
1	&	0.706975	&	0.707074	&	0.707107	\\\hline
$L^\infty$	&	0.000131757	&	3.28838E-05	&		\\\hline

\end{tabular}}								
\end{center}
\caption{\small{Solutions $z^M(t)$ for $J=3$ and $J=4$ for example \ref{Problem2}.}}	
\label{p2Table4}											
\end{table}

\begin{figure}
        \begin{subfigure}[b]{0.45\textwidth}
\includegraphics[scale=0.3]{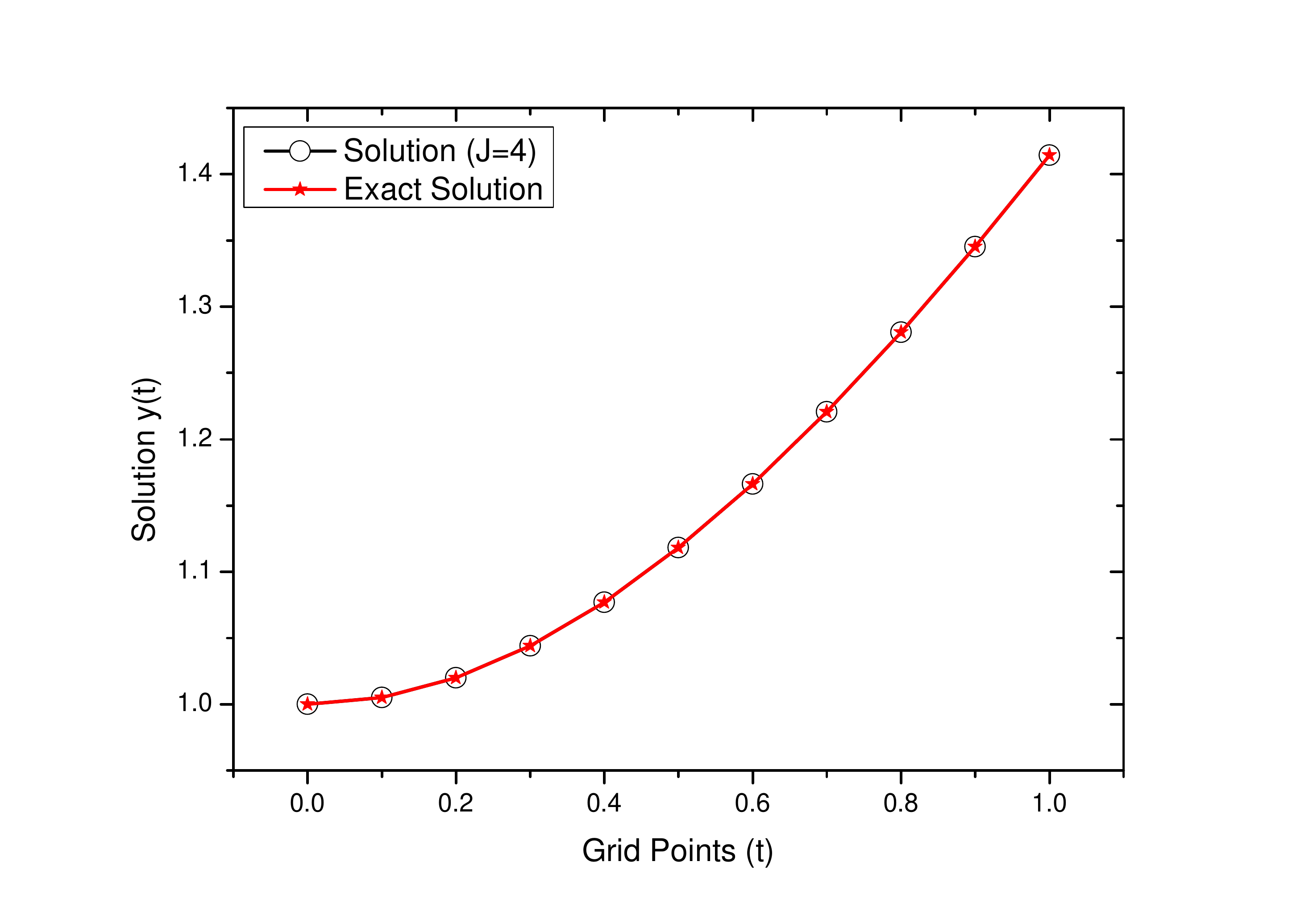}
\subcaption{Graph of $y^M(t)$ for $J=4$  for example \ref{Problem2}.}\label{p2fig5}
        \end{subfigure}\hspace{0.5in}
        \begin{subfigure}[b]{0.45\textwidth}
\includegraphics[scale=0.3]{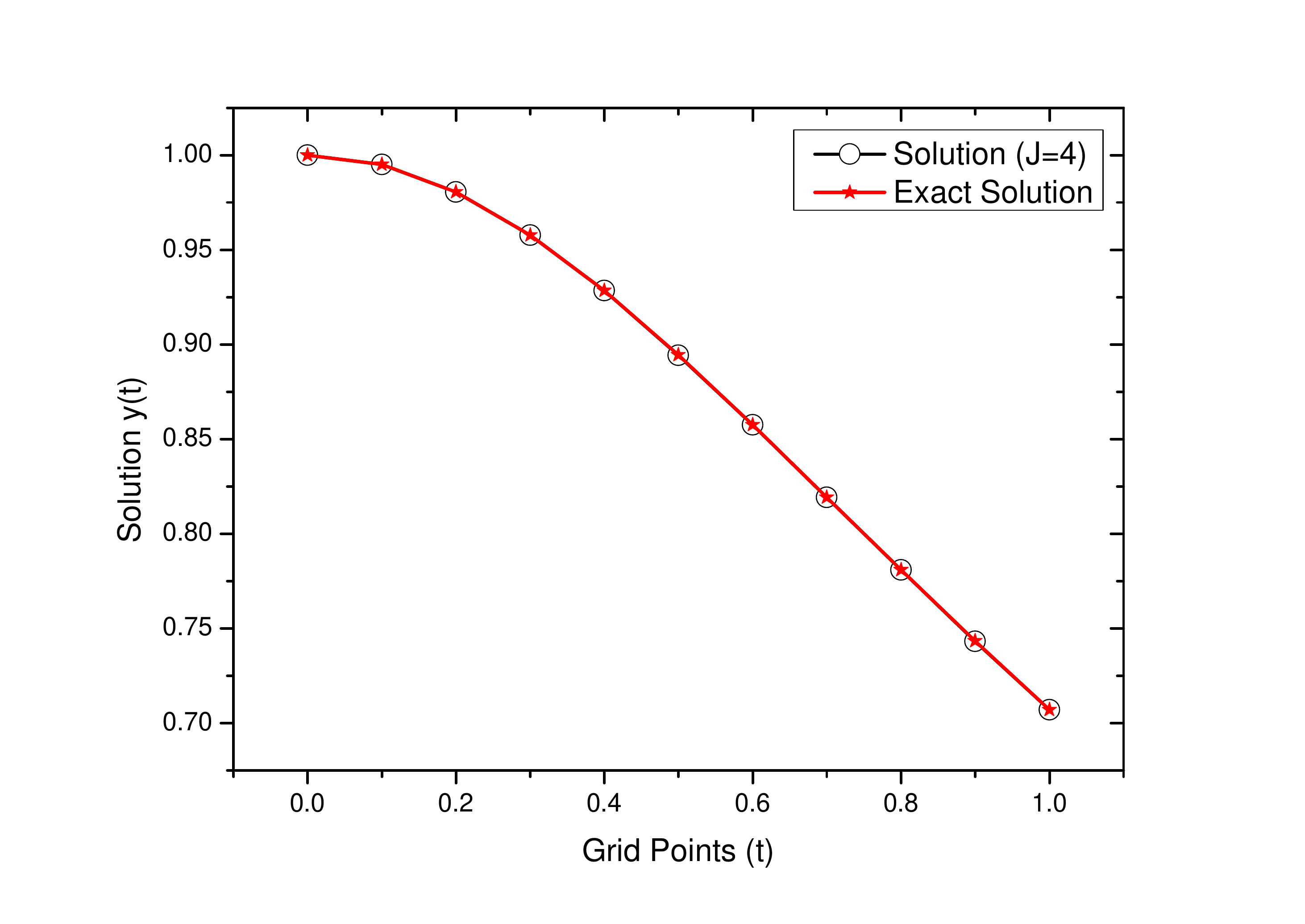}
\subcaption{Graph of $z^M(t)$ for $J=4$  for example \ref{Problem2}.}\label{p2fig6}
        \end{subfigure}%
        \caption{}
        \label{ivpfig3}
\end{figure}

\begin{figure}
        \begin{subfigure}[b]{0.45\textwidth}
\includegraphics[scale=0.3]{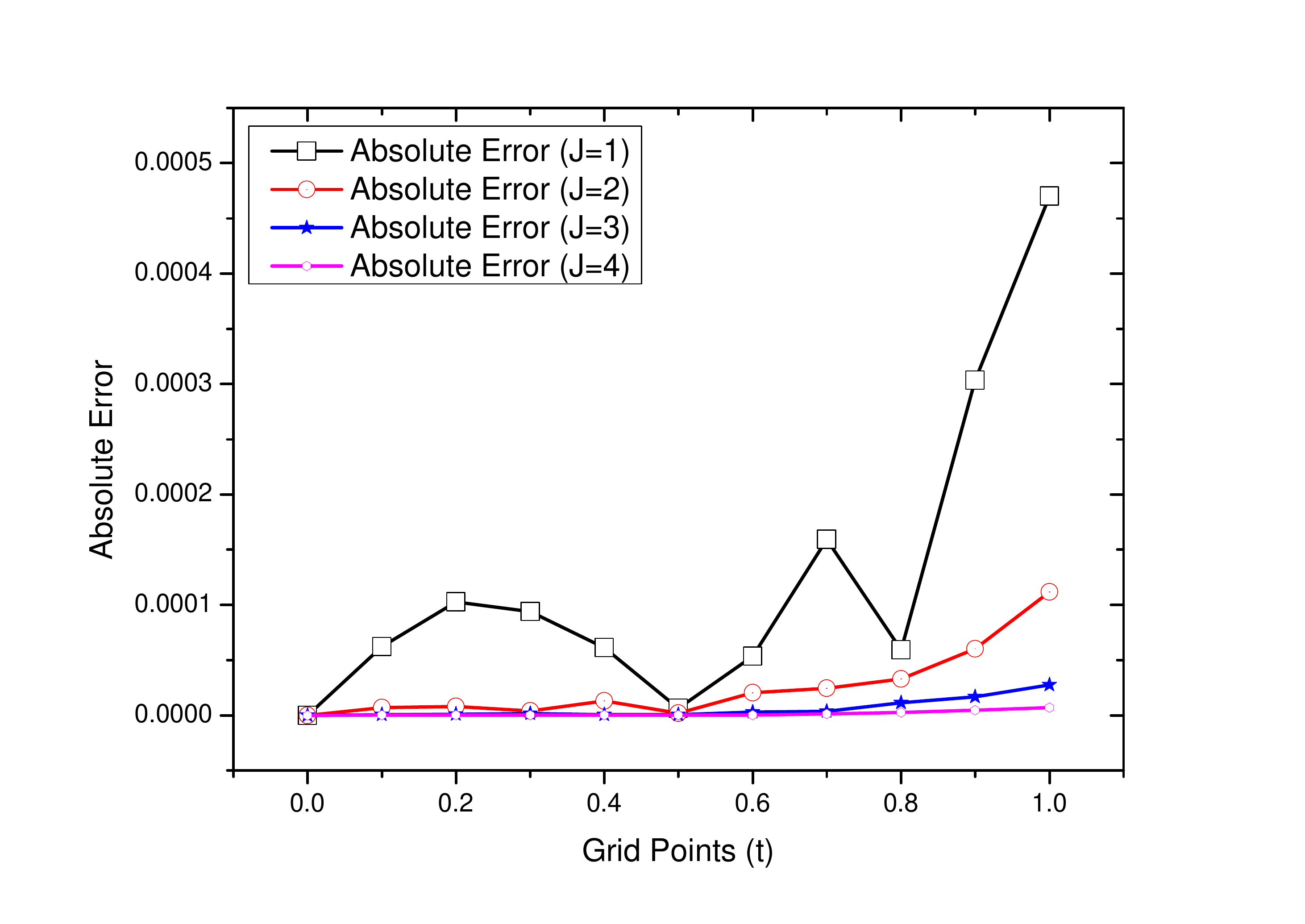}
\subcaption{Graph of absolute errors in $y^M(t)$ for example \ref{Problem2}.}\label{p2fig7}
        \end{subfigure}\hspace{0.5in}
        \begin{subfigure}[b]{0.45\textwidth}
\includegraphics[scale=0.3]{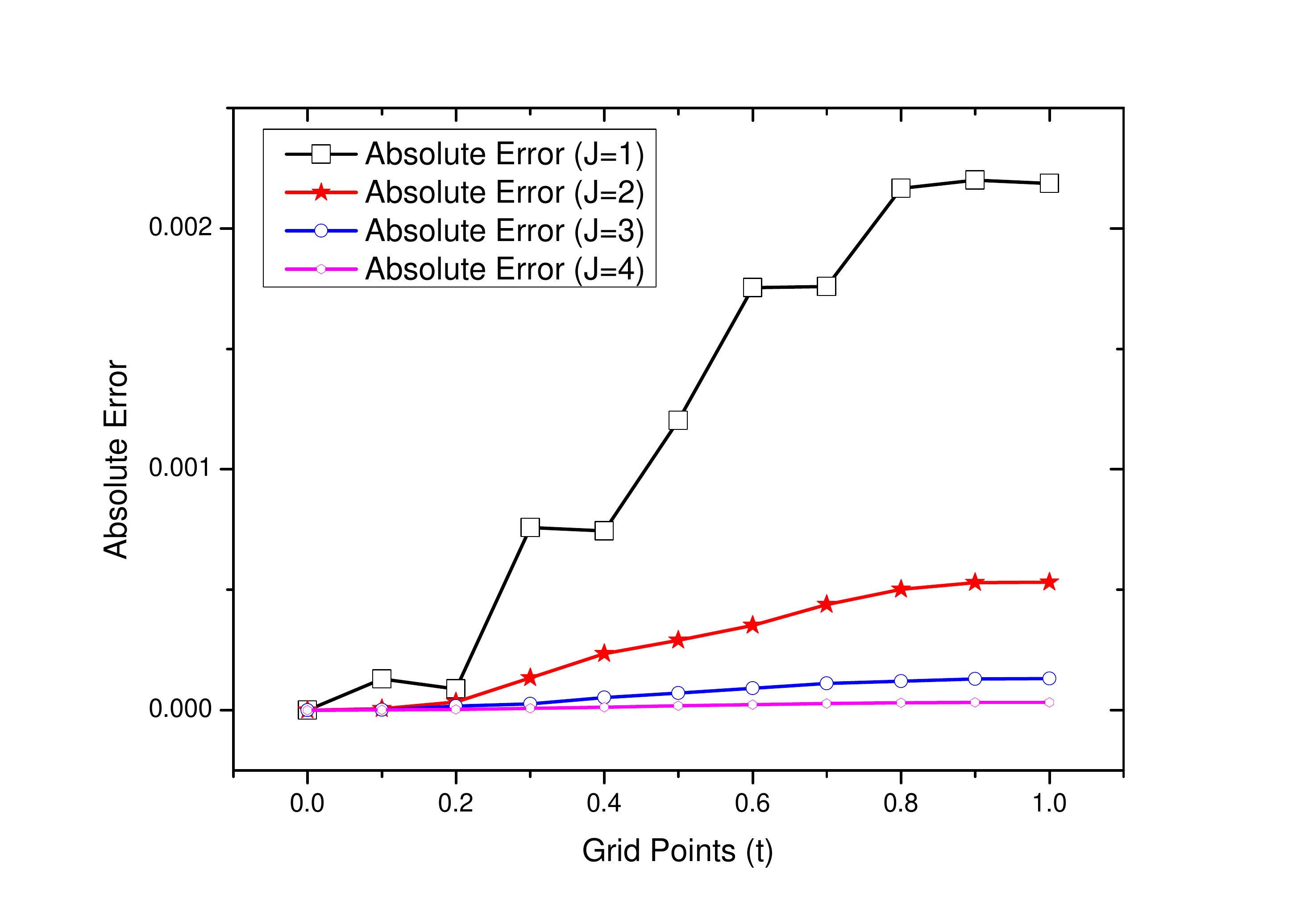}
\subcaption{Graph of absolute errors in $z^M(t)$ for example \ref{Problem2}.}\label{p2fig8}
        \end{subfigure}
        \caption{}
          \label{ivpfig4}
\end{figure}

\section{System of Two Point Nonlinear Singular BVP}\label{Sec2PBVP}

In this section we will develop the solution technique for solving system of nonlinear singular two point boundary value problem. We will also discuss some numerical examples based on it.

\subsection{Method 2 : HWCABVP}\label{method2PBVP}

Consider \eqref{p2eqn1}-\eqref{p2eqn2} with the following boundary conditions,
\begin{equation}\label{p2eq26}
y'(0)=0,~y(1)=\delta_1,
\end{equation} 
\begin{equation}\label{p2eq27}
z'(0)=0,~z(1)=\delta_2,
\end{equation}
where $t \in (0,1)$ and $\delta_1$, $\delta_2$, $k_1\geq 1$, $k_2\geq 1$, $\omega_1<1$, $\omega_2<1$ are real constants.

\begin{theorem}
Consider the system of differential equations \eqref{p2eqn1}-\eqref{p2eqn2} with boundary conditions \eqref{p2eq26}-\eqref{p2eq27}. Let us assume $f_1(t,y,z)$, $f_2(t,y,z)$ be continuous functions in $t$. Let $y(t)$ and $z(t)$ be the solutions of the differential equations \eqref{p2eqn1}-\eqref{p2eqn2}  subject to the initial conditions \eqref{p2eq26}-\eqref{p2eq27}. Then the numerical solutions $y(t)$ and $z(t)$ for differential equations \eqref{p2eqn1}-\eqref{p2eqn2}  using Haar wavelet method are defined as follows,
\begin{equation*}
y(t)=\delta_1+\sum_{i=1}^{2M} a_i (P_{2,i}(t)-P_{2,i}(1)),
\end{equation*}
\begin{equation*}
z(t)=\delta_2+\sum_{i=1}^{2M} b_i (P_{2,i}(t)-P_{2,i}(1)).
\end{equation*}
\end{theorem}

\begin{proof}
Let \eqref{p2eq14}, \eqref{p2eq15} be the solutions of \eqref{p2eqn1}-\eqref{p2eqn2}  where $a_i$, $b_i$ are wavelet coefficients. Now integrate \eqref{p2eq14}, \eqref{p2eq15} two times from $0$ to $t$ we will get \eqref{p2eq16}-\eqref{p2eq19} then apply boundary conditions \eqref{p2eq26} and \eqref{p2eq27} and we get,

\begin{equation}\label{p2eq28}
y'(t)=\sum_{i=1}^{2M} a_i P_{1,i}(t),
\end{equation} 
\begin{equation}\label{p2eq29}
y(t)=\delta_1+\sum_{i=1}^{2M} a_i (P_{2,i}(t)-P_{2,i}(1)),
\end{equation} 
\begin{equation}\label{p2eq30}
z'(t)=\sum_{i=1}^{2M} b_i P_{1,i}(t),
\end{equation} 
\begin{equation}\label{p2eq31}
z(t)=\delta_2+\sum_{i=1}^{2M} b_i (P_{2,i}(t)-P_{2,i}(1)).
\end{equation}
\end{proof}

Now substituting equations \eqref{p2eq14}-\eqref{p2eq15} and \eqref{p2eq28}-\eqref{p2eq31} in \eqref{p2eqn1}-\eqref{p2eqn2}  after discretizing by collocation method, we will get the system of nonlinear equations given as,
\begin{equation} \label{p2eq32}
\Phi^{BVP}_c (a_1,a_2,\cdots,a_{2M})=0, ~~~c=1,2,\cdots,2M,
\end{equation} 
\begin{equation} \label{p2eq33}
\Psi^{BVP}_c (b_1,b_2,\cdots,b_{2M})=0, ~~~c=1,2,\cdots,2M,
\end{equation}
which we solve by Newton-Raphson method to get the wavelet coefficients $a_i$ and $b_i$. Finally substituting the values of $a_i,b_i$ in equations \eqref{p2eq29} and \eqref{p2eq31}, we get the approximate solution obtained by HWCABVP.

\subsection{Convergence Analysis of HWCABVP}\label{five}
Let us consider $y$ and $z$ are the exact solutions of differential equation \eqref{p2eqn1}-\eqref{p2eqn2} with boundary conditions \eqref{p2eq26}-\eqref{p2eq27} and $y^{M}$, $z^{M}$ are the solutions obtained after truncation of the wavelet series. Then
\begin{eqnarray*}
	&&y^{M} = \delta_{1}+\sum_{0}^{2M}a_{i}(p_{2,i}(t)-p_{2,i}(1)),\\
	&&|\tilde{E}_{M1}|=\left|\sum_{j=J+1}^{\infty}\sum_{k=0}^{2^j-1}a_{2^j+k+1}(p_{2,2^j+k+1}(t)-p_{2,2^r+s+1}(1))\right|,
\end{eqnarray*}
\begin{eqnarray*}
	&&z^{M} =\delta_{2}+\sum_{0}^{2M}a_{i}(p_{2,i}(t)-p_{2,i}(1)),\\
	&&|\tilde{E}_{M2}|=\left|\sum_{j=J+1}^{\infty}\sum_{k=0}^{2^j-1}b_{2^j+k+1}(p_{2,2^j+k+1}(t)-p_{2,2^r+s+1}(1))\right|,
\end{eqnarray*}
where $|\tilde{E}_{M1}|=|y^{M}-y|$ and $|\tilde{E}_{M2}|=|z^{M}-z|$. Similar to case of HWCAIVP, Now we define 
\begin{eqnarray}\label{con}
||\tilde{E}_M||_2=||\tilde{E}_{M1}||_2+||\tilde{E}_{M2}||_2.
\end{eqnarray}
\begin{theorem}\label{theorembvp}
	Consider the system of differential equations \eqref{p2eqn1}-\eqref{p2eqn2} subject to boundary conditions \eqref{p2eq26}-\eqref{p2eq27}. Let us assume that  $y'''(t),z'''(t)\in L^2(\mathbb{R})$ are continuous functions on $[0,1]$. Consider $y'''_{r+1}(t)$, $z'''_{r+1}(t)$  are bounded such that $\left|y'''_{r+1}(t)\right|\leq\xi_1$ and $\left|z'''_{r+1}(t)\right|\leq\xi_2$\quad $\forall t \in[0,1]$. Let $\epsilon>0$ be arbitrary small positive number, and if $J>\log_2\left(\sqrt{\frac{2\xi_1}{3\epsilon}}\right)-1$, then $||\tilde{E}_M||_2<\epsilon$. 
\end{theorem}
\begin{proof}
	We will do calculation for $\tilde{E}_{M1}$, for $\tilde{E}_{M2}$ it will follow accordingly.
	Expanding quadrate of $L^2$ norm of error function, we obtain
	\begin{eqnarray*}
		||\tilde{E}_{M1}||_{2}^2= \int_{0}^{1}{\left(\sum_{j=J+1}^{\infty}\sum_{k=0}^{2^j-1}a_{2^j+k+1}(p_{2,2^j+k+1}(t)-p_{2,2^j+k+1}(1))\right)}^2 dt.
	\end{eqnarray*}
	Then we have
	\begin{eqnarray}\label{b}
	||\tilde{E}_{M1}||_{2}^2 = \sum_{j=J+1}^{\infty}\sum_{k=0}^{2^j-1}\sum_{r=J+1}^{\infty}\sum_{s=0}^{2^r-1}a_{2^j+k+1}a_{2^r+s+1}\int_{0}^{1}(p_{2,2^j+k+1}(t)-p_{2,2^j+k+1}(1))(p_{2,2^r+s+1}(t)-p_{2,2^r+s+1}(1))dt.
	\end{eqnarray}
	By the use of definition of Haar wavelet function we can evaluate coefficients $a_{i}$ as
	\begin{eqnarray*}
		&&a_{i}=2^j\int_{0}^{1}y''(t)h_{i}(t)dt,\\
		&&a_{i}=2^j\left[\int_{\eta_{1}}^{\eta_{2}}y''(\zeta)d\zeta-\int_{\eta_{2}}^{\eta_{3}}y''(\zeta)d\zeta\right],\\
		&&a_{i}=2^j[(\eta_{2}-\eta_{1})y''(\zeta_{1})-(\eta_{3}-\eta_{2})y''(\zeta_{2})],\\
	\end{eqnarray*}
	where $\zeta_{1} \in (\eta_{1},\eta_{2})$ and $\zeta_{2} \in (\eta_{2},\eta_{3})$. It follows from Eq. \eqref{p2eq2} that $(\eta_{2}-\eta_{1})=(\eta_{3}-\eta_{2})=1/(2m)=1/(2^{j+1})$, then the above expression of $a_{i}$ reduces to
	\begin{eqnarray*}
		&&a_{i}=\frac{1}{2}[y''(\zeta_{1})-y''(\zeta_{2})]=\frac{1}{2}(\zeta_{1}-\zeta_{2})\frac{dy''}{dt}(\zeta),\:\:\:  \zeta \in (\zeta_{1},\zeta_{2}).
	\end{eqnarray*}
	Let us consider that $\left(\frac{dy''}{dt}\right)$ is bounded such that $\left|\frac{dy''}{dt}\right|\leq\xi_1$, so we have
	\begin{equation}\label{c}
	a_{i} \leq \xi_1\left(\frac{1}{2^{j+1}}\right).
	\end{equation}
	Here we will solve for upper bound of a function $p_{2,i}$ in all subintervals. Since $p_{2,i}(t)=0$ for $t \in [0,\eta_{1}(i)]$. The function $p_{2,i}(t)$ increases monotonically in the interval $t\in[\eta_{1}(i),\eta_{2}(i)]$. Thus $p_{2,i}(t)$ achieves its upper bound at $t= \eta_{2}(i)$ as follows
	\begin{eqnarray*}
		p_{2,i}=p_{2,2^j+k+1} \leq \frac{{[\eta_{2}(i)-\eta_{1}(i)]}^2}{2}=\frac{1}{2}\left(\frac{1}{2^{j+1}}\right)^2,\:\:\: t\in[\eta_{1}(i),\eta_{2}(i)].
	\end{eqnarray*}
	\noindent In the interval $t \in [\eta_{2}(i),\eta_{3}(i)]$ the function $p_{2,i}$ is monotonically increasing if
	\begin{eqnarray}\label{this}
	t \leq \eta_{3},
	\end{eqnarray}
	which is obviously true.
	
	This inequality \eqref{this} can be derived from formulas \eqref{p2eq2} and \eqref{p2eq6} and condition $\frac{dp_{2,i}(t)}{dt}>0$.
	Hence maximum value of $p_{2,i}(t)$ can be obtained by substituting $t=\eta_{3}(i)$ in eq. \eqref{p2eq6} as
	\begin{eqnarray*}
		p_{2,i}(t)=p_{2,2^j+k+1}\leq\left(\frac{1}{2^{j+1}}\right)^2,\:\:\:t\in[\eta_{2}(i),\eta_{3}(i)].
	\end{eqnarray*}
	When $t \in [\eta_{3},1]$ the function $p_{2,i}(t)$ can be expanded as (by eq. \eqref{p2eq6}) (see \cite{VF2001})
	\begin{eqnarray*}
		p_{2,i}(t)=\left(\frac{1}{2^{j+1}}\right)^2.
	\end{eqnarray*}
	The function $p_{2,i}(t)$ increases monotonically in $[0,1]$, since it increases monotonically in every sub interval of $[0,1]$. So upper bound of $p_{2,i}(t)$ in $[0,1]$ is given by
	\begin{eqnarray}\label{d}
	p_{2,i}(t)\leq \left(\frac{1}{2^{j+1}}\right)^2 \:\:\: \forall t \in[0,1] .
	\end{eqnarray}
	Now inserting eq. \eqref{c} in equation \eqref{b} we get
	\begin{eqnarray}\label{f}
	||\tilde{E}_{M1}||_{2}^2 \leq {\xi_1}^2 \sum_{j=J+1}^{\infty}\sum_{k=0}^{2^j-1}\sum_{r=J+1}^{\infty}\sum_{s=0}^{2^r-1}\frac{1}{2^{r+j+2}}\int_{0}^{1}(p_{2,2^j+k+1}(t)-p_{2,2^j+k+1}(1))(p_{2,2^r+s+1}(t)-p_{2,2^r+s+1}(1))dt.
	\end{eqnarray}
	Now since
	\begin{eqnarray}\label{g}
	(p_{2,2^j+k+1}(t)-p_{2,2^j+k+1}(1))\leq |p_{2,2^j+k+1}(t)|+|p_{2,2^j+k+1}(1))|.
	\end{eqnarray}
	Here we have
	\begin{eqnarray}\label{h}
	p_{2,2^j+k+1}(t)\leq\left(\frac{1}{2^{j+1}}\right)^2, \quad p_{2,2^r+s+1}(t) \leq \left (\frac{1}{2^{r+1}}\right)^2\:\:\forall t \in [0,1].
	\end{eqnarray}
	Using eq. \eqref{h} in equation \eqref{g} we get
	\begin{eqnarray}\label{i}
	(p_{2,2^j+k+1}(t)-p_{2,2^j+k+1}(1))\leq 2\left(\frac{1}{2^{j+1}}\right)^2.
	\end{eqnarray}
	Similarly,
	\begin{eqnarray}\label{j}
	(p_{2,2^r+s+1}(t)-p_{2,2^r+s+1}(1))\leq 2\left(\frac{1}{2^{r+1}}\right)^2.
	\end{eqnarray}
	Inserting eq. \eqref{i}, \eqref{j} in equation \eqref{f} we get
	\begin{eqnarray*}
		&&||\tilde{E}_{M1}||_{2}^2 \leq 4{\xi_1}^2 \sum_{j=J+1}^{\infty}\sum_{r=J+1}^{\infty}\left(\frac{1}{2^{j+1}}\right)^3\left(\frac{1}{2^{r+1}}\right)^3 2^j2^r(1-\eta_{1}),
		\\&&||\tilde{E}_{M1}||_{2}^2 \leq \frac{1}{9}\xi_1^2\left(\frac{1}{2^{J+1}}\right)^4.
	\end{eqnarray*}
	Hence,
	\begin{eqnarray*}
		&&||\tilde{E}_{M1}||_{2} \leq \frac{1}{3}\xi_1\left(\frac{1}{2^{J+1}}\right)^2.
	\end{eqnarray*}
	Hence the proof is completed as per arguments given in proof of theorem \ref{theoremivp}.
\end{proof}

\subsection{Numerical Illustration for HWCABVP}

In this section we will discuss two numerical problems based on nonlinear singular two point boundary value problem (Wazwaz et al. \cite{Wazwaz2013ADM}).

\subsubsection{Example 3 (\cite{Wazwaz2013ADM})} \label{Problem3}

Consider system of differential equation \eqref{p2eqn1}-\eqref{p2eqn2} with boundary conditions \eqref{p2eq26}-\eqref{p2eq27} with $\delta_1=1-2\log_e2$, $\delta_2=1+2\log_e2$, $k_1=5$, $k_2=3$, $\omega_1=-5$, $\omega_2=-3$, $f_1(t,y(t),z(t))=8\Big(exp(y-1)+2 exp(-\frac{z-1}{2})\Big)$ and  $f_2(t,y(t),z(t))=-8\Big(exp(-(z-1))+exp(\frac{y-1}{2})\Big)$. Applying solution method (subsection \ref{method2PBVP}) to solve this problem, we get system of non-linear equations. Thus we arrive at \eqref{p2eq32}-\eqref{p2eq33}.  To solve the non-linear equations we use Newton Raphson method to calculate wavelet coefficients ($a_i$) and ($b_i$). After calculating wavelet coefficients ($a_i$) and ($b_i$) we get our required solution. Here we are computing absolute error described in subsubsection $\ref{Problem1}$. The exact solution $\tilde{y}(t)$ and $\tilde{z}(t)$ of this problem are $1-2\log_e(1+t^2)$ and $1+2\log_e(1+t^2)$, respectively.

For initial guess $[0.8,0.8,\cdots,0.8]$ computed solutions for $y^M(t)$ and $z^M(t)$ is given in Table \ref{p2Table5} and Table \ref{p2Table6} respectively for $J=3$ and $J=4$. Graph for $y^M(t)$ and $z^M(t)$ with exact solution is given in figure \ref{p2fig9} and figure \ref{p2fig10} respectively for $J=4$. Graph of absolute errors in computation of $y^M(t)$ and $z^M(t)$ is given in figure \ref{p2fig11} and figure \ref{p2fig12} respectively for $J=1$, $J=2$, $J=3$ and $J=4$.

Small perturbations in initial guesses, say, $[0.82,0.82,\cdots,0.82]$ and $[0.78,0.78,\cdots,0.78]$ does affected the final solution, which shows the stability of the method.

\begin{table}[H]											 
\centering											
\begin{center}											
\resizebox{7cm}{2.5cm}{											
\begin{tabular}	{|c | c|  c|  c|}		
\hline
$t$	&	$J=3$	&	$J=4$	&	Exact Solution	\\\hline
0	&	1.00147	&	1.00037	&	1	\\\hline
0.1	&	0.981545	&	0.980458	&	0.980099	\\\hline
0.2	&	0.922895	&	0.921893	&	0.921559	\\\hline
0.3	&	0.828852	&	0.827944	&	0.827645	\\\hline
0.4	&	0.704169	&	0.703414	&	0.70316	\\\hline
0.5	&	0.554528	&	0.553916	&	0.553713	\\\hline
0.6	&	0.38565	&	0.385183	&	0.385031	\\\hline
0.7	&	0.202866	&	0.202554	&	0.202448	\\\hline
0.8	&	0.0108728	&	0.0106725	&	0.0106075	\\\hline
0.9	&	-0.186541	&	-0.186624	&	-0.186654	\\\hline
1	&	-0.386294	&	-0.386294	&	-0.386294	\\\hline
$L^\infty$	&	0.00147404	&	0.000367459	&		\\\hline

\end{tabular}}								
\end{center}
\caption{\small{Solutions $y^M(t)$ for $J=3$ and $J=4$ for example \ref{Problem3}.}}	
\label{p2Table5}											
\end{table}

\begin{table}[H]											 
\centering											
\begin{center}											
\resizebox{7cm}{2.5cm}{											
\begin{tabular}	{|c | c|  c|  c|}		
\hline
$t$	&	$J=3$	&	$J=4$	&	Exact Solution	\\\hline
0	&	0.998354	&	0.99959	&	1	\\\hline
0.1	&	1.01828	&	1.0195	&	1.0199	\\\hline
0.2	&	1.07693	&	1.07806	&	1.07844	\\\hline
0.3	&	1.17098	&	1.17201	&	1.17236	\\\hline
0.4	&	1.29567	&	1.29655	&	1.29684	\\\hline
0.5	&	1.44532	&	1.44605	&	1.44629	\\\hline
0.6	&	1.61422	&	1.61478	&	1.61497	\\\hline
0.7	&	1.79703	&	1.79742	&	1.79755	\\\hline
0.8	&	1.98905	&	1.98931	&	1.98939	\\\hline
0.9	&	2.1865	&	2.18661	&	2.18665	\\\hline
1	&	2.38629	&	2.38629	&	2.38629	\\\hline
$L^\infty$	&	0.00164587	&	0.000410165	&		\\\hline

\end{tabular}}								
\end{center}
\caption{\small{Solutions $z^M(t)$ for $J=3$ and $J=4$ for example \ref{Problem3}.}}	
\label{p2Table6}											
\end{table}

\begin{figure}
        \begin{subfigure}[b]{0.45\textwidth}
\includegraphics[scale=0.3]{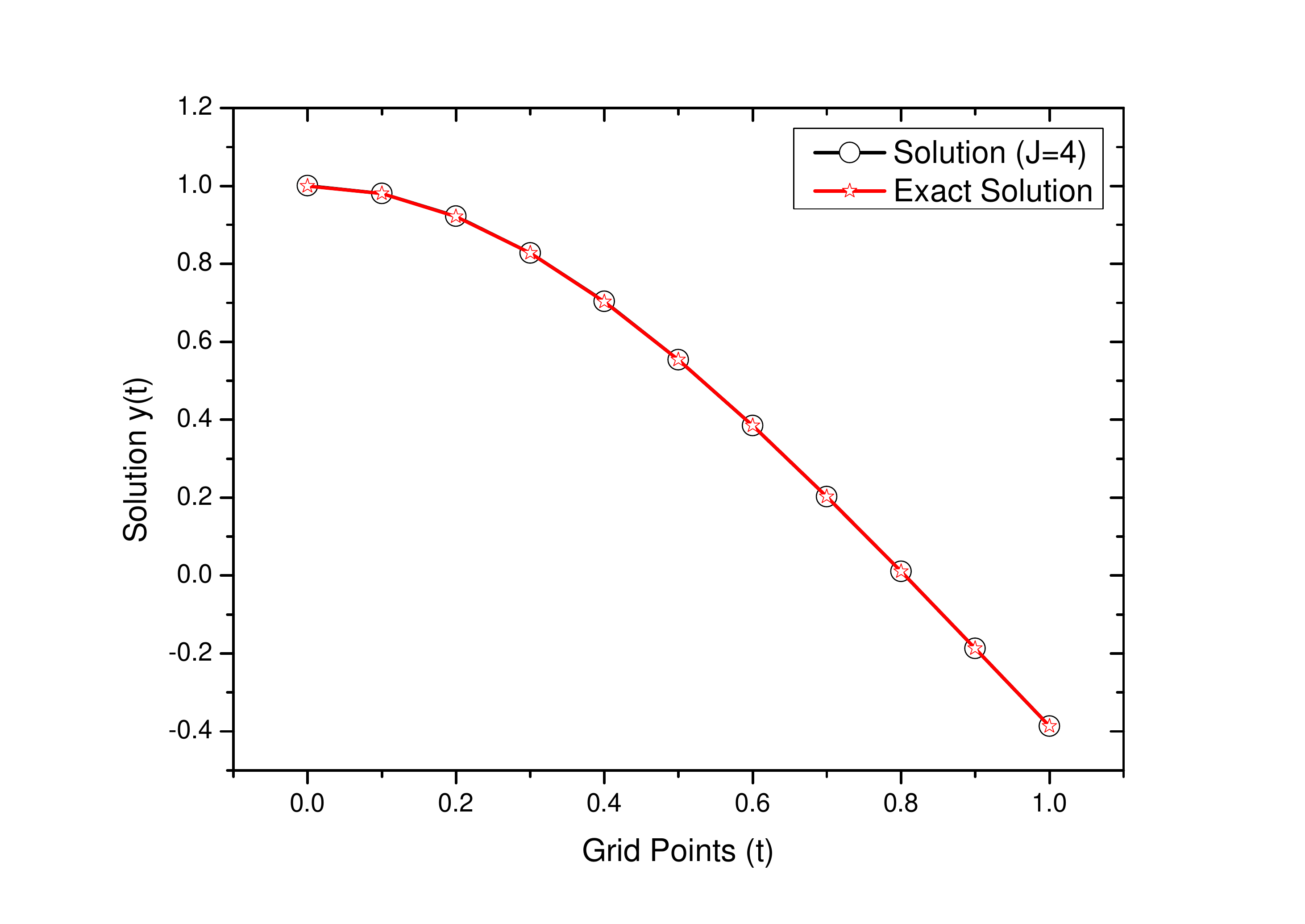}
\subcaption{Graph of $y^M(t)$ for $J=4$ for example \ref{Problem3}.}\label{p2fig9}
        \end{subfigure}\hspace{0.5in}
        \begin{subfigure}[b]{0.45\textwidth}
\includegraphics[scale=0.3]{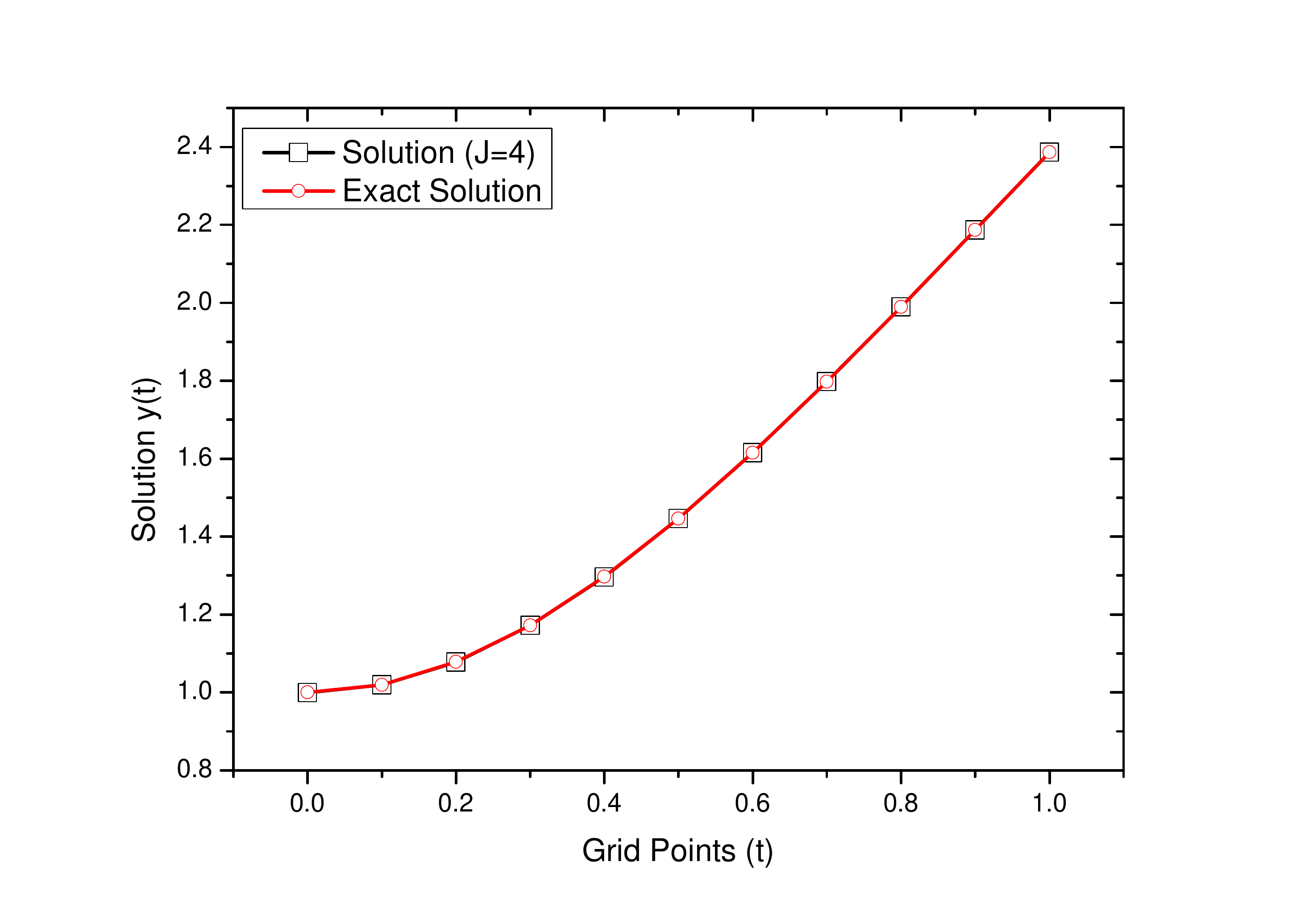}
\subcaption{Graph of $z^M(t)$ for $J=4$ for example \ref{Problem3}}.\label{p2fig10}
        \end{subfigure}
        \caption{}
        \label{bvpfig1}
\end{figure}

\begin{figure}
        \begin{subfigure}[b]{0.45\textwidth}
\includegraphics[scale=0.3]{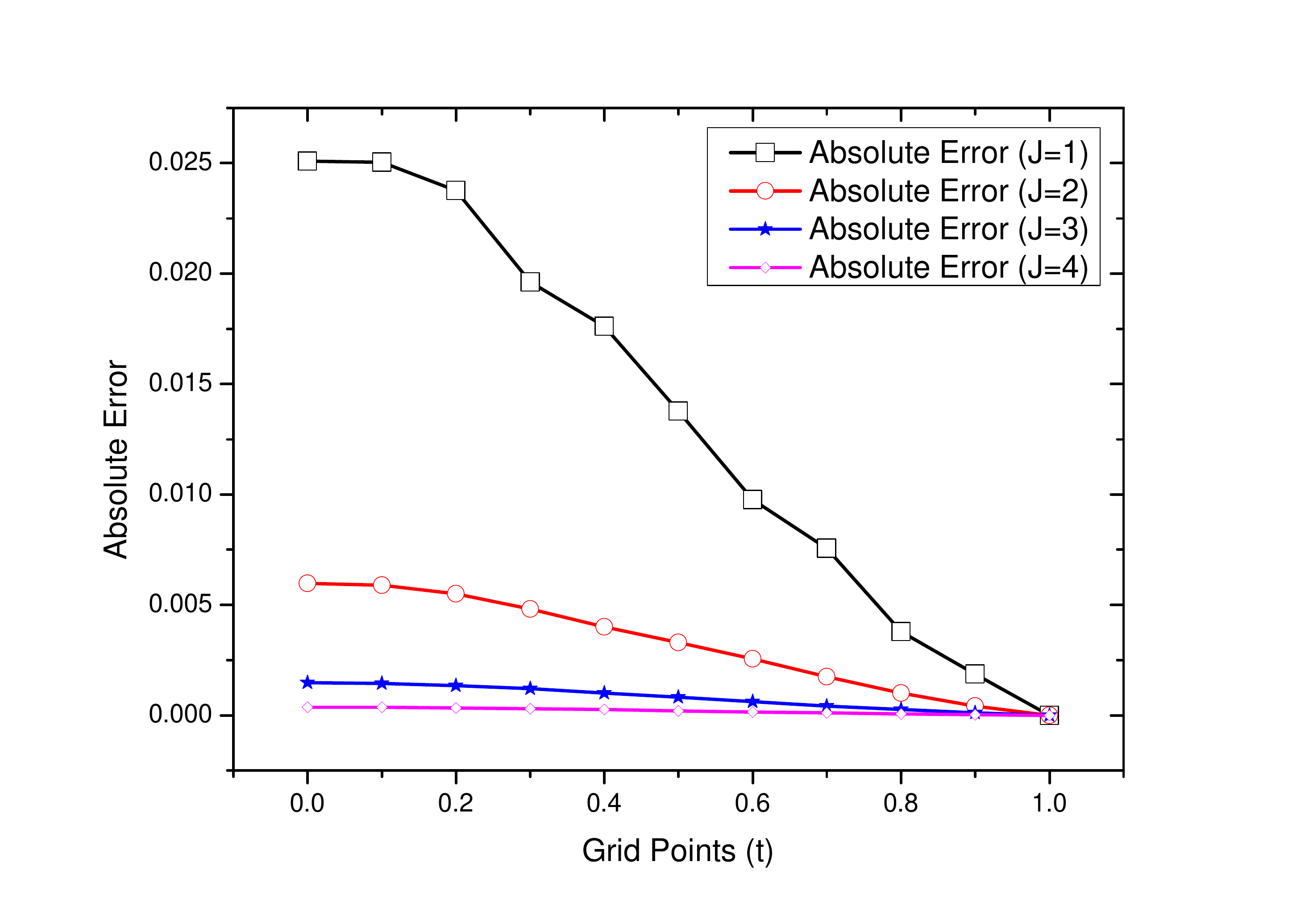}
\subcaption{Graph of absolute errors in $y^M(t)$ for example \ref{Problem3}.}\label{p2fig11}
        \end{subfigure}\hspace{0.5in}
        \begin{subfigure}[b]{0.45\textwidth}
\includegraphics[scale=0.3]{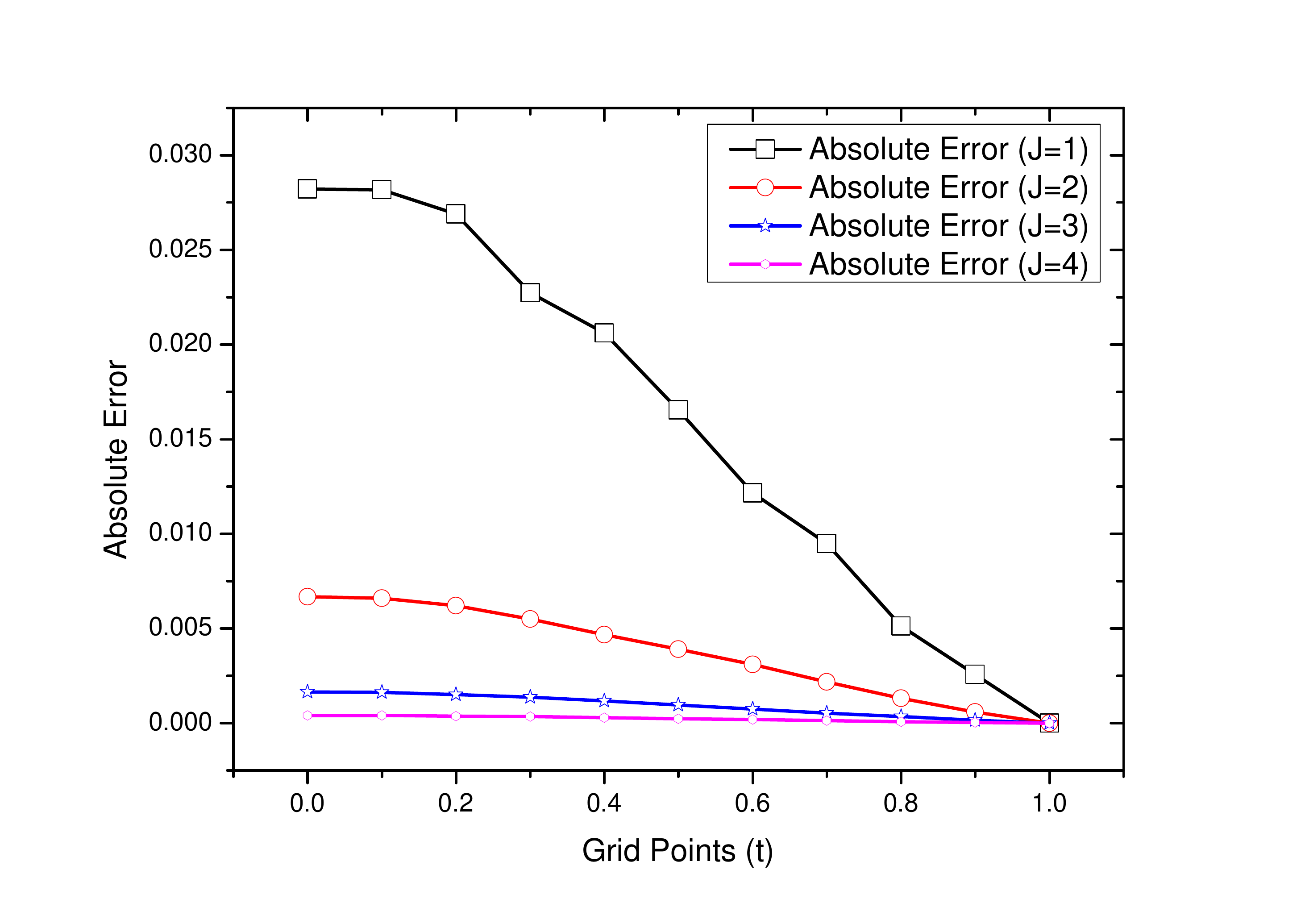}
\subcaption{Graph of absolute errors in $z^M(t)$ for example \ref{Problem3}.}\label{p2fig12}
        \end{subfigure}

        \caption{}        \label{bvpfig2}
\end{figure}

\subsubsection{Example 4 (\cite{Wazwaz2013ADM})} \label{Problem4}

Consider the system of differenital equations \eqref{p2eqn1}-\eqref{p2eqn2} with boundary conditions \eqref{p2eq26}-\eqref{p2eq27} with $\delta_1=\frac{1}{\mathrm{e}}$, $\delta_2=\mathrm{e}$, $k_1=8$, $k_2=4$, $\omega_1=-8$, $\omega_2=-4$, $f_1(t,y(t),z(t))=18y-4y \log_ez$ and $f_2(t,y(t),z(t))=4z\log_ey-10z$. This problem is same as subsubsection $\ref{Problem3}$. Here we are computing absolute error described in subsubsection $\ref{Problem1}$. The exact solution $\tilde{y}(t)$ and $\tilde{z}(t)$ of this problem are $\exp(-t^2)$ and $\exp(t^2)$, respectively.

For initial guess $[0.3,0.3,\cdots,0.3]$ computed solutions for $y^M(t)$ and $z^M(t)$ is given in Table \ref{p2Table7} and Table \ref{p2Table8} respectively for $J=3$ and $J=4$. Graph for $y^M(t)$ and $z^M(t)$ with exact solution is given in figure \ref{p2fig13} and figure \ref{p2fig14} respectively for $J=4$. Graph of absolute errors in computation of $y^M(t)$ and $z^M(t)$ is given in figure \ref{p2fig15} and figure \ref{p2fig16} respectively for $J=1$, $J=2$, $J=3$ and $J=4$.

Varying initial guesses to  $[0.35,0.35,\cdots,0.35]$ and $[0.28,0.28,\cdots,0.28]$, does not change the computed, significantly. This shows that the method is stable.

\begin{table}[H]											 
\centering											
\begin{center}											
\resizebox{7cm}{2.5cm}{											
\begin{tabular}	{|c | c|  c|  c|}		
\hline
$t$	&	$J=3$	&	$J=4$	&	Exact Solution	\\\hline
0	&	1.00075	&	1.00019	&	1	\\\hline
0.1	&	0.990793	&	0.990235	&	0.99005	\\\hline
0.2	&	0.961486	&	0.960964	&	0.960789	\\\hline
0.3	&	0.914573	&	0.91409	&	0.913931	\\\hline
0.4	&	0.852692	&	0.852283	&	0.852144	\\\hline
0.5	&	0.779256	&	0.778915	&	0.778801	\\\hline
0.6	&	0.698032	&	0.697763	&	0.697676	\\\hline
0.7	&	0.612869	&	0.612688	&	0.612626	\\\hline
0.8	&	0.527451	&	0.527331	&	0.527292	\\\hline
0.9	&	0.444924	&	0.444876	&	0.444858	\\\hline
1	&	0.367879	&	0.367879	&	0.367879	\\\hline
$L^\infty$	&	0.000754755	&	0.00018869	&		\\\hline

\end{tabular}}								
\end{center}
\caption{\small{Solutions $y^M(t)$ for $J=3$ and $J=4$ for example \ref{Problem4}.}}	
\label{p2Table7}											
\end{table}

\begin{table}[H]											 
\centering											
\begin{center}											
\resizebox{7cm}{2.5cm}{											
\begin{tabular}	{|c | c|  c|  c|}		
\hline
$t$	&	$J=3$	&	$J=4$	&	Exact Solution	\\\hline
0	&	1.00051	&	1.00013	&	1	\\\hline
0.1	&	1.01055	&	1.01018	&	1.01005	\\\hline
0.2	&	1.0413	&	1.04093	&	1.04081	\\\hline
0.3	&	1.09465	&	1.09429	&	1.09417	\\\hline
0.4	&	1.17395	&	1.17362	&	1.17351	\\\hline
0.5	&	1.28442	&	1.28412	&	1.28403	\\\hline
0.6	&	1.43369	&	1.43341	&	1.43333	\\\hline
0.7	&	1.63256	&	1.63238	&	1.63232	\\\hline
0.8	&	1.89673	&	1.89653	&	1.89648	\\\hline
0.9	&	2.24798	&	2.24794	&	2.24791	\\\hline
1	&	2.71828	&	2.71828	&	2.71828	\\\hline
$L^\infty$	&	0.000505366	&	0.00012642	&		\\\hline

\end{tabular}}								
\end{center}
\caption{\small{Solutions for $z^M(t)$ for $J=3$ and $J=4$ for example \ref{Problem4}.}}	
\label{p2Table8}											
\end{table}

\begin{figure}
        \begin{subfigure}[b]{0.45\textwidth}
\includegraphics[scale=0.3]{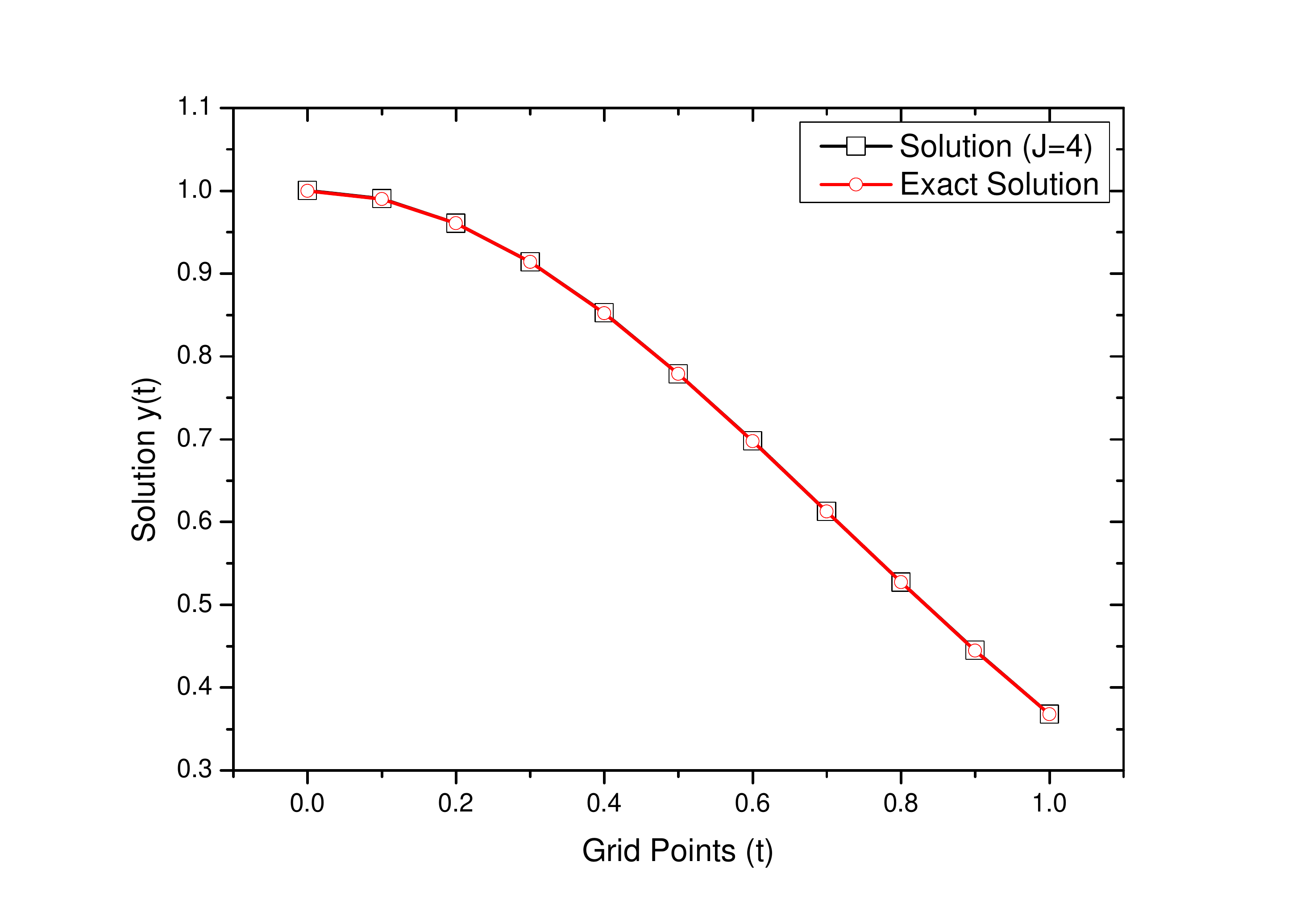}
\caption{Graph of $y^M(t)$ for $J=4$ for example \ref{Problem4}.}\label{p2fig13}
        \end{subfigure}\hspace{0.5in}
        \begin{subfigure}[b]{0.45\textwidth}
\includegraphics[scale=0.3]{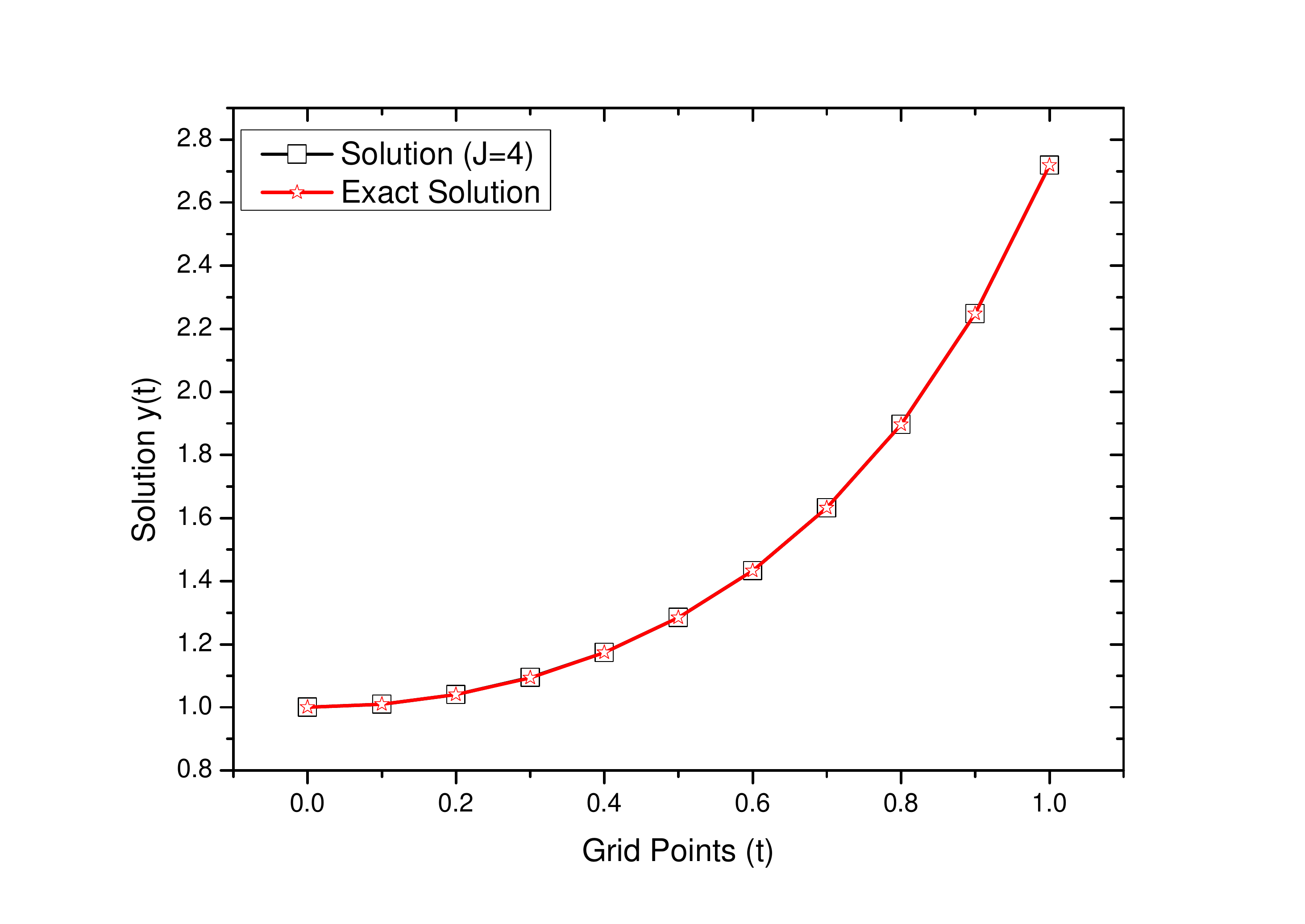}
\caption{Graph of $z^M(t)$ for $J=4$ for example \ref{Problem4}.}\label{p2fig14}
        \end{subfigure}
        \caption{}
        \label{bvpfig3}
\end{figure}

\begin{figure}
        \begin{subfigure}[b]{0.45\textwidth}
\includegraphics[scale=0.3]{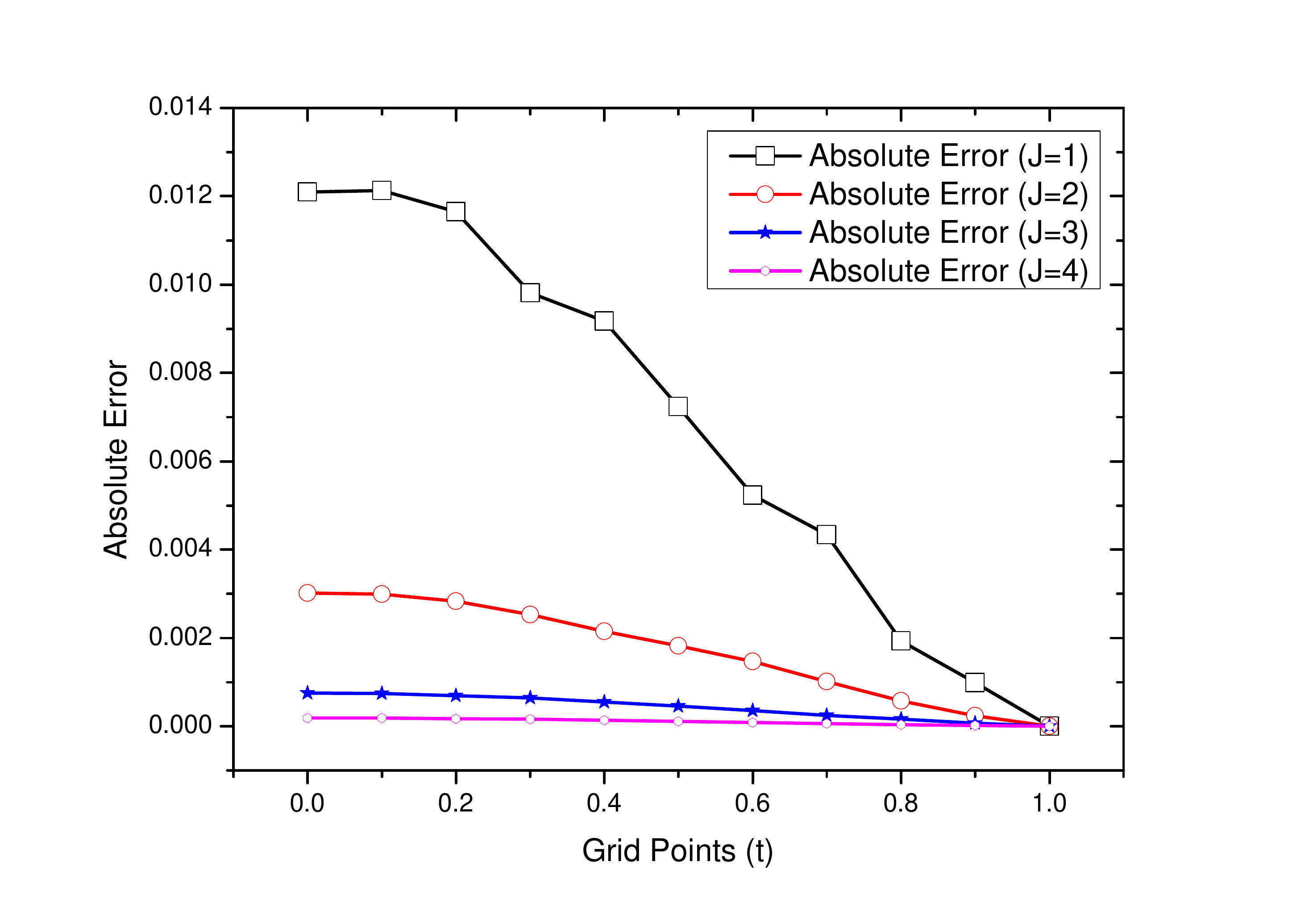}
\caption{Graph of absolute errors in $y^M(t)$ for example \ref{Problem4}.}\label{p2fig15}
        \end{subfigure}\hspace{0.5in}
        \begin{subfigure}[b]{0.45\textwidth}
\includegraphics[scale=0.3]{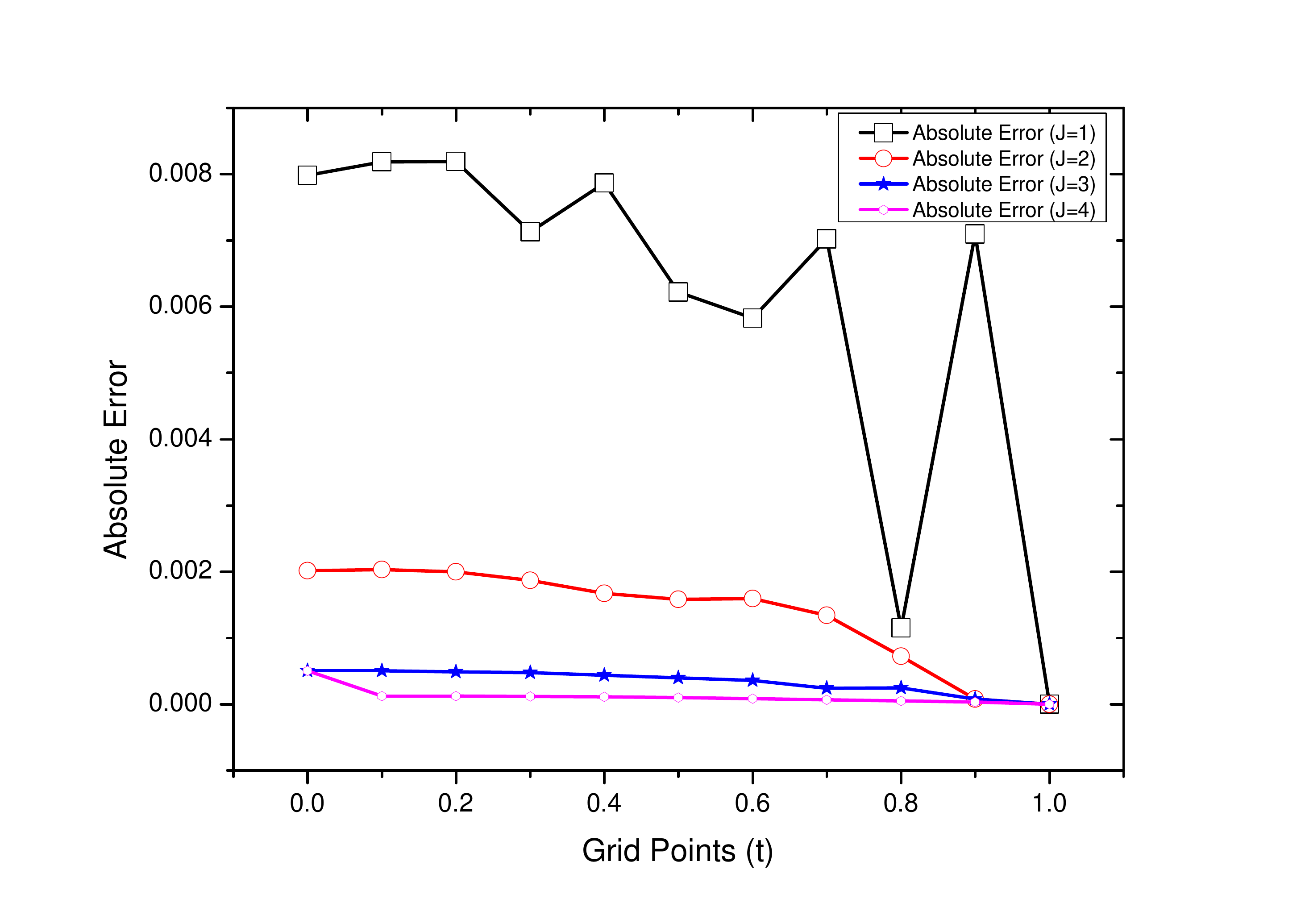}
\caption{Graph of absolute errors in $z^M(t)$ for example \ref{Problem4}.}\label{p2fig16}
        \end{subfigure}
        \caption{}
     \label{bvpfig4}
\end{figure}

\section{System of Nonlinear Singular Four Point BVP}\label{Sec4pBVP}

In this section we will develop the method for solving system of nonlinear singular four point boundary value problem. We also study some numerical examples based on it. These test examples are also considered by Barnwal et al. \cite{Barnwal2019}. We also compare our results with the result of Barnwal et al. \cite{Barnwal2019}.

\subsection{Method 3 :  HWCA4PTBVP}\label{method4PBVP}

Consider the system of differential equations \eqref{p2eqn1}-\eqref{p2eqn2} subject to four-point boundary conditions,
\begin{equation}\label{p2eq34}
y(0)=0,~y(1)=n_1z(v_1),
\end{equation} 
\begin{equation}\label{p2eq35}
z(0)=0,~z(1)=n_2y(v_2),
\end{equation}
where $t \in (0,1)$, $n_1$, $n_2$, $v_1$, $v_2$ $\in (0,1)$ and $0\leq k_1<1$, $0\leq k_2$, $\omega_1<1$, $\omega_2<1$ are real constants.

\begin{theorem}
Consider the system of differential equations  \eqref{p2eqn1}-\eqref{p2eqn2} with boundary conditions \eqref{p2eq34}-\eqref{p2eq35}. Let us assume $f_1(t,y,z)$, $f_2(t,y,z)$ be continuous functions in $t,y,z$. Let $y(t)$ and $z(t)$ be the solutions of the differential equations \eqref{p2eqn1}-\eqref{p2eqn2} subject to the four point boundary conditions \eqref{p2eq34}-\eqref{p2eq35}. Then the numerical solutions $y(t)$ and $z(t)$ for differential equations \eqref{p2eqn1}-\eqref{p2eqn2} using HWCA4PTBVP are defined as follows,
\begin{equation*}
y(t)=\frac{n_1 t}{1-{n_1 n_2 v_1 v_2}} \Big[n_2 v_1 \sum_{i=1}^{2M} a_i [P_{2,i}(v_2)-v_2 P_{2,i}(1)]+\sum_{i=1}^{2M} b_i [P_{2,i}(v_1)-v_1 P_{2,i}(1)]\Big]+\sum_{i=1}^{2M} a_i [P_{2,i}(t)-tP_{2,i}(1)],
\end{equation*}
\begin{equation*}
z(t)=\frac{n_2 t}{1-{n_1 n_2 v_1 v_2}} \Big[\sum_{i=1}^{2M} a_i [P_{2,i}(v_2)-v_2 P_{2,i}(1)]+v_2 n_1 \sum_{i=1}^{2M} b_i [P_{2,i}(v_1)-v_1 P_{2,i}(1)]\Big]+\sum_{i=1}^{2M} b_i [P_{2,i}(t)-tP_{2,i}(1)].
\end{equation*}
\end{theorem}

\begin{proof}
Let \eqref{p2eq14}, \eqref{p2eq15} be the solutions of \eqref{p2eqn1}-\eqref{p2eqn2} where $a_i$, $b_i$ are wavelet coefficients. Now integrate \eqref{p2eq14}-\eqref{p2eq15} two times from $0$ to $t$ we will get \eqref{p2eq16}-\eqref{p2eq19} then apply boundary conditions \eqref{p2eq34}-\eqref{p2eq35} and we get,
\begin{equation}\label{p2eq36}
y'(t)=\frac{n_1}{1-{n_1 n_2 v_1 v_2}} \Big[n_2 v_1 \sum_{i=1}^{2M} a_i [P_{2,i}(v_2)-v_2 P_{2,i}(1)]+\sum_{i=1}^{2M} b_i [P_{2,i}(v_1)-v_1 P_{2,i}(1)]\Big]+\sum_{i=1}^{2M} a_i [P_{1,i}(t)-P_{2,i}(1)],
\end{equation}
\begin{equation}\label{p2eq37}
y(t)=\frac{n_1 t}{1-{n_1 n_2 v_1 v_2}} \Big[n_2 v_1 \sum_{i=1}^{2M} a_i [P_{2,i}(v_2)-v_2 P_{2,i}(1)]+\sum_{i=1}^{2M} b_i [P_{2,i}(v_1)-v_1 P_{2,i}(1)]\Big]+\sum_{i=1}^{2M} a_i [P_{2,i}(t)-tP_{2,i}(1)],
\end{equation}
\begin{equation}\label{p2eq38}
z'(t)=\frac{n_2}{1-{n_1 n_2 v_1 v_2}} \Big[\sum_{i=1}^{2M} a_i [P_{2,i}(v_2)-v_2 P_{2,i}(1)]+v_2 n_1 \sum_{i=1}^{2M} b_i [P_{2,i}(v_1)-v_1 P_{2,i}(1)]\Big]+\sum_{i=1}^{2M} b_i [P_{1,i}(t)-P_{2,i}(1)],
\end{equation}
\begin{equation}\label{p2eq39}
z(t)=\frac{n_2 t}{1-{n_1 n_2 v_1 v_2}} \Big[\sum_{i=1}^{2M} a_i [P_{2,i}(v_2)-v_2 P_{2,i}(1)]+v_2 n_1 \sum_{i=1}^{2M} b_i [P_{2,i}(v_1)-v_1 P_{2,i}(1)]\Big]+\sum_{i=1}^{2M} b_i [P_{2,i}(t)-tP_{2,i}(1)].
\end{equation}
\end{proof}

Now substituting these equations \eqref{p2eq14}-\eqref{p2eq15} and \eqref{p2eq36}-\eqref{p2eq39} in \eqref{p2eqn1}-\eqref{p2eqn2} after discretizing by collocation method, we will get the system of nonlinear equations given as,
\begin{equation} \label{p2eq40}
\Phi^{4PtBVP}_c (a_1,a_2,\cdots,a_{2M})=0, ~~~c=1,2,\cdots,2M,
\end{equation} 
\begin{equation} \label{p2eq41}
\Psi^{4PtBVP}_c (b_1,b_2,\cdots,b_{2M})=0, ~~~c=1,2,\cdots,2M,
\end{equation}
which we solve by Newton Raphson method to get the wavelet coefficients $a_i$ and $b_i$ and substitute them in \eqref{p2eq37} and \eqref{p2eq39} to get approximate HWCA4PTBVP solutions.

\subsection{Convergence Analysis of HWCA4PBVP}\label{4pbvp}
Define
\begin{equation}\label{p2eq371}
y(t)=\frac{n_1 t}{1-{n_1 n_2 v_1 v_2}} \left[n_2 v_1 \sum_{i=1}^{\infty} a_i [P_{2,i}(v_2)-v_2 P_{2,i}(1)]+\sum_{i=1}^{\infty} b_i [P_{2,i}(v_1)-v_1 P_{2,i}(1)]\right]+\sum_{i=1}^{\infty} a_i \left[P_{2,i}(t)-tP_{2,i}(1)\right],
\end{equation}
\begin{equation}\label{p2eq391}
z(t)=\frac{n_2 t}{1-{n_1 n_2 v_1 v_2}} \left[\sum_{i=1}^{\infty} a_i [P_{2,i}(v_2)-v_2 P_{2,i}(1)]+v_2 n_1 \sum_{i=1}^{\infty} b_i [P_{2,i}(v_1)-v_1 P_{2,i}(1)]\right]+\sum_{i=1}^{\infty} b_i \left[P_{2,i}(t)-tP_{2,i}(1)\right].
\end{equation}
\begin{equation}\label{p2eq371M}
y^M(t)=\frac{n_1 t}{1-{n_1 n_2 v_1 v_2}} \left[n_2 v_1 \sum_{i=1}^{2M} a_i [P_{2,i}(v_2)-v_2 P_{2,i}(1)]+\sum_{i=1}^{2M} b_i [P_{2,i}(v_1)-v_1 P_{2,i}(1)]\right]+\sum_{i=1}^{2M} a_i \left[P_{2,i}(t)-tP_{2,i}(1)\right],
\end{equation}
\begin{equation}\label{p2eq391M}
z^M(t)=\frac{n_2 t}{1-{n_1 n_2 v_1 v_2}} \left[\sum_{i=1}^{2M} a_i [P_{2,i}(v_2)-v_2 P_{2,i}(1)]+v_2 n_1 \sum_{i=1}^{2M} b_i [P_{2,i}(v_1)-v_1 P_{2,i}(1)]\right]+\sum_{i=1}^{2M} b_i \left[P_{2,i}(t)-tP_{2,i}(1)\right].
\end{equation}
Then,
\begin{multline*}
	|\overline{E}_{M1}|=\left|\frac{n_1 t}{1-{n_1 n_2 v_1 v_2}} \left[n_2 v_1 \sum_{j=J+1}^{\infty}\sum_{k=0}^{2^j-1}a_{2^j+k+1}\left[P_{2,2^j+k+1}(v_2)-v_2 P_{2,2^j+k+1}(1)\right]\right.\right.\\+\left.\sum_{j=J+1}^{\infty}\sum_{k=0}^{2^j-1}b_{2^j+k+1}\left[P_{2,2^j+k+1}(v_1)-v_1P_{2,2^j+k+1}(1)\right]\right]\\
	\left.+\sum_{j=J+1}^{\infty}\sum_{k=0}^{2^j-1}a_{2^j+k+1}\left[P_{2,2^j+k+1}(t)-tP_{2,2^j+k+1}(1)\right]\right|,
\end{multline*}
\begin{multline*}
|\overline{E}_{M2}|=\left|\frac{n_2 t}{1-{n_1 n_2 v_1 v_2}} \left[\sum_{j=J+1}^{\infty}\sum_{k=0}^{2^j-1}a_{2^j+k+1} \left[P_{2,2^j+k+1}(v_2)-v_2 P_{2,2^j+k+1}(1)\right]\right.\right.
\\\left.+v_2 n_1 \sum_{j=J+1}^{\infty}\sum_{k=0}^{2^j-1}b_{2^j+k+1}\left[P_{2,2^j+k+1}(v_1)-v_1 P_{2,2^j+k+1}(1)\right]\right]\\\left.+\sum_{i=1}^{2M} b_{2^j+k+1} \left[P_{2,2^j+k+1}(t)-tP_{2,2^j+k+1}(1)\right]\right|.
\end{multline*}
where $|\overline{E}_{M1}|=|y^{M}-y|$ and $|\overline{E}_{M2}|=|z^{M}-z|$. Here again we define total error as 
\begin{eqnarray}\label{con}
||\overline{E}_M||_2=||\overline{E}_{M1}||_2+||\overline{E}_{M2}||_2.
\end{eqnarray}
\begin{theorem}\label{theorem4pbvp}
Consider the system of differential equation \eqref{p2eqn1}-\eqref{p2eqn2} with four point boundary conditions \eqref{p2eq34}-\eqref{p2eq35}. Let us assume that  $y'''(t),z'''(t)\in L^2(\mathbb{R})$ are continuous functions on $[0,1]$. Consider $y'''_{r+1}(t)$, $z'''_{r+1}(t)$  are bounded such that $\left|y'''_{r+1}(t)\right|\leq\xi_1$ and $\left|z'''_{r+1}(t)\right|\leq\xi_2$\quad $\forall t \in[0,1]$. Let $\epsilon>0$ be arbitrary small positive number and if $J>\log_2\sqrt{\frac{2 \tilde{C}}{\epsilon}}-1$, then $||\overline{E}_M||_2<\epsilon$, where $\tilde{C}$ is real number which depends on $n_1,n_2,v_2,v_2,\xi_1,\xi_2$. \end{theorem}
\begin{proof}
We perform the calculation for $\overline{E}_{M1}$, for $\overline{E}_{M2}$ it will follow accordingly. Now, by the use of definition of Haar wavelet function we can evaluate coefficients $a_{i}$ as
	\begin{eqnarray*}
		&&a_{i}=2^j\int_{0}^{1}y''(t)h_{i}(t)dt,\\
		&&a_{i}=2^j\left[\int_{\eta_{1}}^{\eta_{2}}y''(\zeta)d\zeta-\int_{\eta_{2}}^{\eta_{3}}y''(\zeta)d\zeta\right],\\
		&&a_{i}=2^j[(\eta_{2}-\eta_{1})y''(\zeta_{1})-(\eta_{3}-\eta_{2})y''(\zeta_{2})],\\
	\end{eqnarray*}
similarly coefficients $b_{i}$ as
\begin{eqnarray*}
	b_{i}=2^j[(\eta_{2}-\eta_{1})z''(\zeta_{1})-(\eta_{3}-\eta_{2})z''(\zeta_{2})],\\
\end{eqnarray*}
	where $\zeta_{1} \in (\eta_{1},\eta_{2})$ and $\zeta_{2} \in (\eta_{2},\eta_{3})$. It follows from Eq. \eqref{p2eq2} that $(\eta_{2}-\eta_{1})=(\eta_{3}-\eta_{2})=1/(2m)=1/(2^{j+1})$, then the above expression of $a_{i}$, $b_{i}$ reduces to
	\begin{eqnarray*}
		&&a_{i}=\frac{1}{2}[y''(\zeta_{1})-y''(\zeta_{2})]=\frac{1}{2}(\zeta_{1}-\zeta_{2})\frac{dz''}{dt}(\zeta),\:\:\:  \zeta \in (\zeta_{1},\zeta_{2}),\\
	&&b_{i}=\frac{1}{2}(\zeta_{1}-\zeta_{2})\frac{dy''}{dt}(\zeta),\:\:\:  \zeta \in (\zeta_{1},\zeta_{2}).
\end{eqnarray*}
	Let us consider that $\left(\frac{dy''}{dt}\right)$ is bounded such that $\left|\frac{dy''}{dt}\right|\leq\xi_1$, so we have
	\begin{equation}\label{c3}
	a_{i} \leq \xi_1\left(\frac{1}{2^{j+1}}\right).
	\end{equation}
	similarly let us consider that $\left(\frac{dz''}{dt}\right)$ is bounded such that $\left|\frac{dz''}{dt}\right|\leq\xi_2$, so we have
	\begin{equation}\label{c31}
	b_{i} \leq \xi_2\left(\frac{1}{2^{j+1}}\right).
	\end{equation}
	Here we will solve for upper bound of a function $p_{2,i}$ in all subintervals. Since $p_{2,i}(t)=0$ for $t \in [0,\eta_{1}(i)]$. The function $p_{2,i}(t)$ increases monotonically in the interval $t\in[\eta_{1}(i),\eta_{2}(i)]$. Thus $p_{2,i}(t)$ achieves its upper bound at $t= \eta_{2}(i)$ as follows
	\begin{eqnarray*}
		p_{2,i}=p_{2,2^j+k+1} \leq \frac{{[\eta_{2}(i)-\eta_{1}(i)]}^2}{2}=\frac{1}{2}\left(\frac{1}{2^{j+1}}\right)^2,\:\:\: t\in[\eta_{1}(i),\eta_{2}(i)].
	\end{eqnarray*}
	\noindent In the interval $t \in [\eta_{2}(i),\eta_{3}(i)]$ the function $p_{2,i}$ is monotonically increasing if
	\begin{eqnarray}\label{this3}
	t \leq \eta_{3},
	\end{eqnarray}
	which is obviously true.
	
	This inequality \eqref{this3} can be derived from formulas \eqref{p2eq2} and \eqref{p2eq6} and condition $\frac{dp_{2,i}(t)}{dt}>0$.
	Hence maximum value of $p_{2,i}(t)$ can be obtained by substituting $t=\eta_{3}(i)$ in eq. \eqref{p2eq6} as
	\begin{eqnarray*}
		p_{2,i}(t)=p_{2,2^j+k+1}\leq\left(\frac{1}{2^{j+1}}\right)^2,\:\:\:t\in[\eta_{2}(i),\eta_{3}(i)].
	\end{eqnarray*}
	When $t \in [\eta_{3},1]$ the function $p_{2,i}(t)$ can be expanded as (by eq. \eqref{p2eq6}) (see \cite{VF2001})
	\begin{eqnarray*}
		p_{2,i}(t)=\left(\frac{1}{2^{j+1}}\right)^2.
	\end{eqnarray*}
	The function $p_{2,i}(t)$ increases monotonically in $[0,1]$, since it increases monotonically in every sub interval of $[0,1]$. So upper bound of $p_{2,i}(t)$ in $[0,1]$ is given by
	\begin{eqnarray}\label{d3}
	p_{2,i}(t)\leq \left(\frac{1}{2^{j+1}}\right)^2 \:\:\: \forall t \in[0,1] .
	\end{eqnarray}
	Expanding quadrate of $L^2$ norm of error function, we obtain
	\begin{multline}
||\overline{E}_{M1}||_{2}^2= \displaystyle\int_{0}^{1}\left(\frac{n_1 t}{1-{n_1 n_2 v_1 v_2}}\right. \left[n_2 v_1 \sum_{j=J+1}^{\infty}\sum_{k=0}^{2^j-1}a_{2^j+k+1}\left[P_{2,2^j+k+1}(v_2)-v_2P_{2,2^j+k+1}(1)\right]\right.+\\\sum_{j=J+1}^{\infty}\sum_{k=0}^{2^j-1}b_{2^j+k+1}\left.\left[P_{2,2^j+k+1}(v_1)-v_1P_{2,2^j+k+1}(1)\right]\right]\\+\sum_{j=J+1}^{\infty}\sum_{k=0}^{2^j-1}a_{2^j+k+1}a_i\left.\left[P_{2,2^j+k+1}(t)-tP_{2,2^j+k+1}(1)\right]\right)^2 dt.
	\end{multline}
	Expanding and calculating the integral we have
	\begin{multline}\label{b3}
	||\overline{E}_{M1}||_{2}^2 =\frac{n_1^2n_2^2v_1^2}{3(1-n_1 n_2 v_1 v_2)^2}\sum_{j=J+1}^{\infty}\sum_{k=0}^{2^j-1}\sum_{r=J+1}^{\infty}\sum_{s=0}^{2^r-1}a_{2^j+k+1}a_{2^r+s+1}(p_{2,2^j+k+1}(v_2)-v_2p_{2,2^j+k+1}(1))
	\\(p_{2,2^r+s+1}(v_2)-v_2p_{2,2^r+s+1}(1))\\
	+\frac{n_1^2}{3(1-n_1 n_2 v_1 v_2)^2}\sum_{j=J+1}^{\infty}\sum_{k=0}^{2^j-1}\sum_{r=J+1}^{\infty}\sum_{s=0}^{2^r-1}b_{2^j+k+1}b_{2^r+s+1}
	\\(p_{2,2^j+k+1}(v_1)-v_1p_{2,2^j+k+1}(1))(p_{2,2^r+s+1}(v_1)-v_1p_{2,2^r+s+1}(1))\\
	+\sum_{j=J+1}^{\infty}\sum_{k=0}^{2^j-1}\sum_{r=J+1}^{\infty}\sum_{s=0}^{2^r-1}a_{2^r+s+1}a_{2^j+k+1}
	\\\int_{0}^{1}(p_{2,2^j+k+1}(t)-tp_{2,2^r+s+1}(1))(p_{2,2^r+s+1}(t)-tp_{2,2^r+s+1}(1))dt\\
	+\frac{2n_1^2n_2v_1}{3(1-n_1 n_2 v_1 v_2)^2}\sum_{j=J+1}^{\infty}\sum_{k=0}^{2^j-1}\sum_{r=J+1}^{\infty}
	\\\sum_{s=0}^{2^r-1}a_{2^j+k+1}b_{2^r+s+1}(p_{2,2^j+k+1}(v_2)-v_2p_{2,2^j+k+1}(1))(p_{2,2^r+s+1}(v_1)-v_1p_{2,2^r+s+1}(1))\\
	+\frac{2n_1}{1-n_1n_2v_1v_2}\sum_{j=J+1}^{\infty}\sum_{k=0}^{2^j-1}\sum_{r=J+1}^{\infty}
	\\\sum_{s=0}^{2^r-1}b_{2^j+k+1}a_{2^r+s+1}(p_{2,2^j+k+1}(v_1)-v_1p_{2,2^j+k+1}(1))\Big(\frac{1}{2}p_{2,2^r+s+1}(t)-\frac{1}{3}p_{2,2^r+s+1}(1)\Big)
	\\+\frac{2n_1n_2v_1}{1-n_1n_2v_1v_2}\sum_{j=J+1}^{\infty}\sum_{k=0}^{2^j-1}\sum_{r=J+1}^{\infty}\sum_{s=0}^{2^r-1}b_{2^j+k+1}a_{2^r+s+1}(p_{2,2^r+s+1}(v_2)-v_2p_{2,2^r+s+1}(1))\\
	\Big(\frac{1}{2}p_{2,2^j+k+1}(t)-\frac{1}{3}p_{2,2^j+k+1}(1)\Big).
	\end{multline}
	
	Now inserting equation \eqref{c3}, \eqref{c31} and \eqref{d3} in equation \eqref{b3} and following the calculation similar to the proof of theorem \ref{theorembvp}, we get
	\begin{eqnarray*}
		&&||\overline{E}_{M1}||_{2} \leq \tilde{C}\left(\frac{1}{2^{J+1}}\right)^2
	\end{eqnarray*}
where $\tilde{C}$ is a constant which depends on $n_1,n_2,v_1,v_2,\xi_1,\xi_2$. Rest of the proof can be completed as we did for theorem \ref{theoremivp}.
\end{proof}

\subsection{Numerical Illustration for HWCA4PTBVP}

In this section we will discuss two numerical problems based on system of nonlinear singular four point boundary value problem (\cite{Barnwal2019}).

\subsubsection{Example 5 (\cite{Barnwal2019})} \label{Problem5}

Consider the system of singular nonlinear differential equations \eqref{p2eqn1}-\eqref{p2eqn2}  with four point boundary conditions \eqref{p2eq34}-\eqref{p2eq35}with $n_1=1$, $n_2=1$, $v_1=\frac{1}{2}$, $v_2=\frac{1}{3}$, $k_1=\frac{1}{2}$, $k_2=\frac{1}{2}$, $\omega_1=-\frac{1}{2}$, $\omega_2=-\frac{1}{2}$, $f_1(t,y(t),z(t))=\Big(\frac{99}{35}t-\frac{1}{2}+\Big(t^2-\frac{66}{35}t^3+\frac{1089}{1225}t^4\Big)z-y^2z\Big)$ and  $f_2(t,y(t),z(t))=\Big(-\frac{24}{35t}+\frac{64}{1225}t^5-\frac{2112}{42875}t^6-yz^2\Big)$. Applying solution method (subsection \ref{method4PBVP}) to solve this problem, we get system of non-linear equations. Thus we arrive at \eqref{p2eq40}-\eqref{p2eq41}.  To solve the non-linear equations we use Newton Raphson method to calculate wavelet coefficients ($a_i$) and ($b_i$). After calculating wavelet coefficients ($a_i$) and ($b_i$) we get our required solution by substituting $a_i$ and $b_i$ in \eqref{p2eq37} and \eqref{p2eq39}. For computation of absolute error see subsubsection $\ref{Problem1}$. The exact solution $\tilde{y}(t)$ and $\tilde{z}(t)$ of this problem are $-\frac{33}{35}t^2+t$ and $\frac{8}{35}t^2$ respectively.

For initial guess $[1.25,1.25,\cdots,1.25]$ computed solutions for $y^M(t)$ and $z^M(t)$ is given in Table \ref{p2Table9} and Table \ref{p2Table10}, respectively for $J=3$ and $J=4$. Graph for $y^M(t)$ and $z^M(t)$ with exact solution is given in figure \ref{p2fig17} and figure \ref{p2fig18}, respectively for $J=4$. Graph of absolute errors in computation of $y^M(t)$ and $z^M(t)$ is given in figure \ref{p2fig19} and figure \ref{p2fig20}, respectively for $J=1$, $J=2$, $J=3$ and $J=4$.

If we change initial guesses to $[1.29,1.29,\cdots,1.29]$ and $[1.21,1.21,\cdots,1.21]$ we observe that solution does not change significantly. 

\begin{table}[H]											 
\centering											
\begin{center}											
\resizebox{9cm}{2.5cm}{											
\begin{tabular}	{|c | c|  c|  c|  c|}		
\hline
$t$	&	$J=3$	&	$J=4$	&	Exact Solution	&	Successive Iter. Tech. \cite{Barnwal2019}	\\\hline
0	&	0	&	0	&	0	&	0	\\\hline
0.1	&	0.0905714	&	0.0905714	&	0.0905714	&	0.09057158	\\\hline
0.2	&	0.162286	&	0.162286	&	0.162286	&	0.16228593	\\\hline
0.3	&	0.215143	&	0.215143	&	0.215143	&	0.21514312	\\\hline
0.4	&	0.249143	&	0.249143	&	0.249143	&	0.24914315	\\\hline
0.5	&	0.264286	&	0.264286	&	0.264286	&	0.26428604	\\\hline
0.6	&	0.260571	&	0.260571	&	0.260571	&	0.26057177	\\\hline
0.7	&	0.238	&	0.238	&	0.238	&	0.23800035	\\\hline
0.8	&	0.196571	&	0.196571	&	0.196571	&	0.19657177	\\\hline
0.9	&	0.136286	&	0.136286	&	0.136286	&	0.13628605	\\\hline
1	&	0.0571429	&	0.0571429	&	0.0571429	&	0.05714317	\\\hline
$L^\infty$	&	5.55112E-17	&	8.32667E-17	&		& 3.4978 E-7		\\\hline

\end{tabular}}								
\end{center}
\caption{\small{Solutions $y^M(t)$ for $J=3$ and $J=4$  for example \ref{Problem5}.}}	
\label{p2Table9}											
\end{table}

\begin{table}[H]											 
\centering											
\begin{center}											
\resizebox{9cm}{2.5cm}{											
\begin{tabular}	{|c | c|  c|  c|  c|}		
\hline
$t$	&	$J=3$	&	$J=4$	&	Exact Solution	&	Successive Iter. Tech. \cite{Barnwal2019}	\\\hline
0	&	0	&	0	&	0	&	0	\\\hline
0.1	&	0.00228571	&	0.00228571	&	0.00228571	&	0.00228586	\\\hline
0.2	&	0.00914286	&	0.00914286	&	0.00914286	&	0.00914306	\\\hline
0.3	&	0.0205714	&	0.0205714	&	0.0205714	&	0.02057168	\\\hline
0.4	&	0.0365714	&	0.0365714	&	0.0365714	&	0.03657172	\\\hline
0.5	&	0.0571429	&	0.0571429	&	0.0571429	&	0.05714317	\\\hline
0.6	&	0.0822857	&	0.0822857	&	0.0822857	&	0.08228604	\\\hline
0.7	&	0.112	&	0.112	&	0.112	&	0.11200033	\\\hline
0.8	&	0.146286	&	0.146286	&	0.146286	&	0.14628604	\\\hline
0.9	&	0.185143	&	0.185143	&	0.185143	&	0.18514316	\\\hline
1	&	0.228571	&	0.228571	&	0.228571	&	0.2285717	\\\hline
$L^\infty$	&	5.55112E-17	&	5.55112E-17	&		&	3.3431E-7	\\\hline

\end{tabular}}								
\end{center}
\caption{\small{Solutions $z^M(t)$ for $J=3$, $J=4$ for example \ref{Problem5}.}}	
\label{p2Table10}											
\end{table}

\begin{figure}
        \begin{subfigure}[b]{0.45\textwidth}
\includegraphics[scale=0.3]{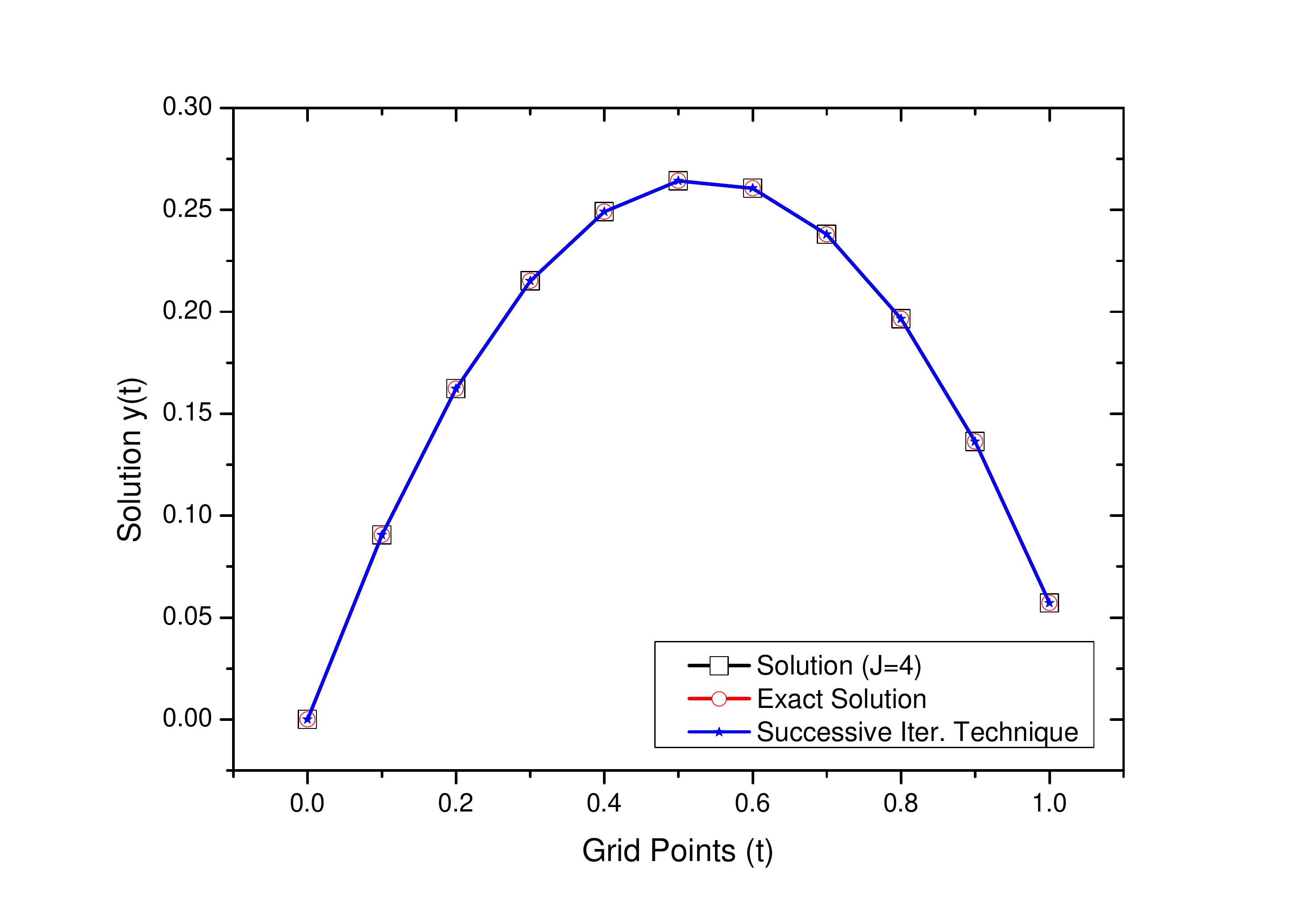}
\caption{Graph of $y^M(t)$ for $J=4$  for example \ref{Problem5}.}\label{p2fig17}
        \end{subfigure}\hspace{0.5in}
        \begin{subfigure}[b]{0.45\textwidth}
\includegraphics[scale=0.3]{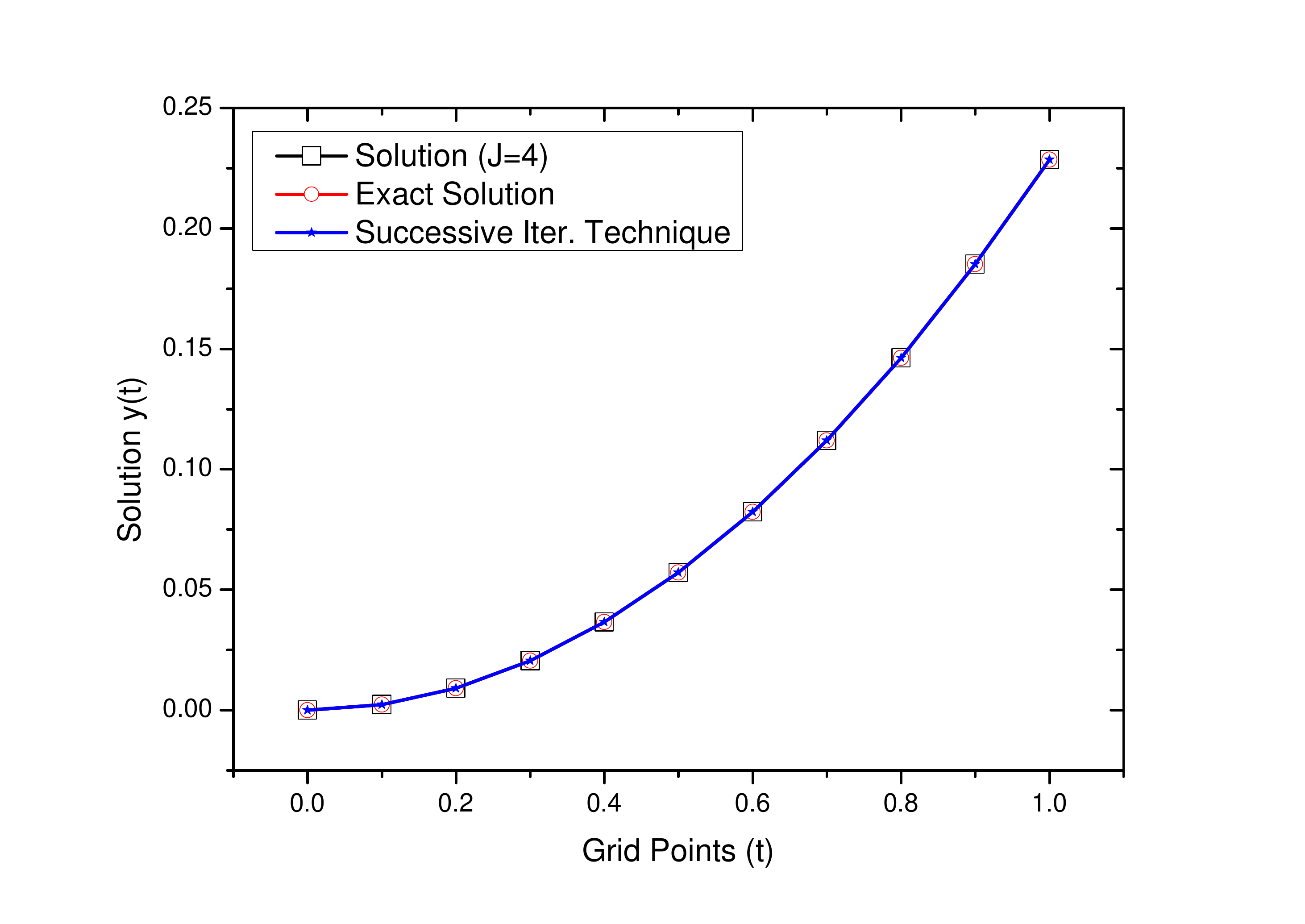}
\caption{Graph of $z^M(t)$ for $J=4$  for example \ref{Problem5}.}\label{p2fig18}
        \end{subfigure}
        \caption{}
        \label{4pbvpfig1}
\end{figure}

\begin{figure}
        \begin{subfigure}[b]{0.45\textwidth}
\includegraphics[scale=0.3]{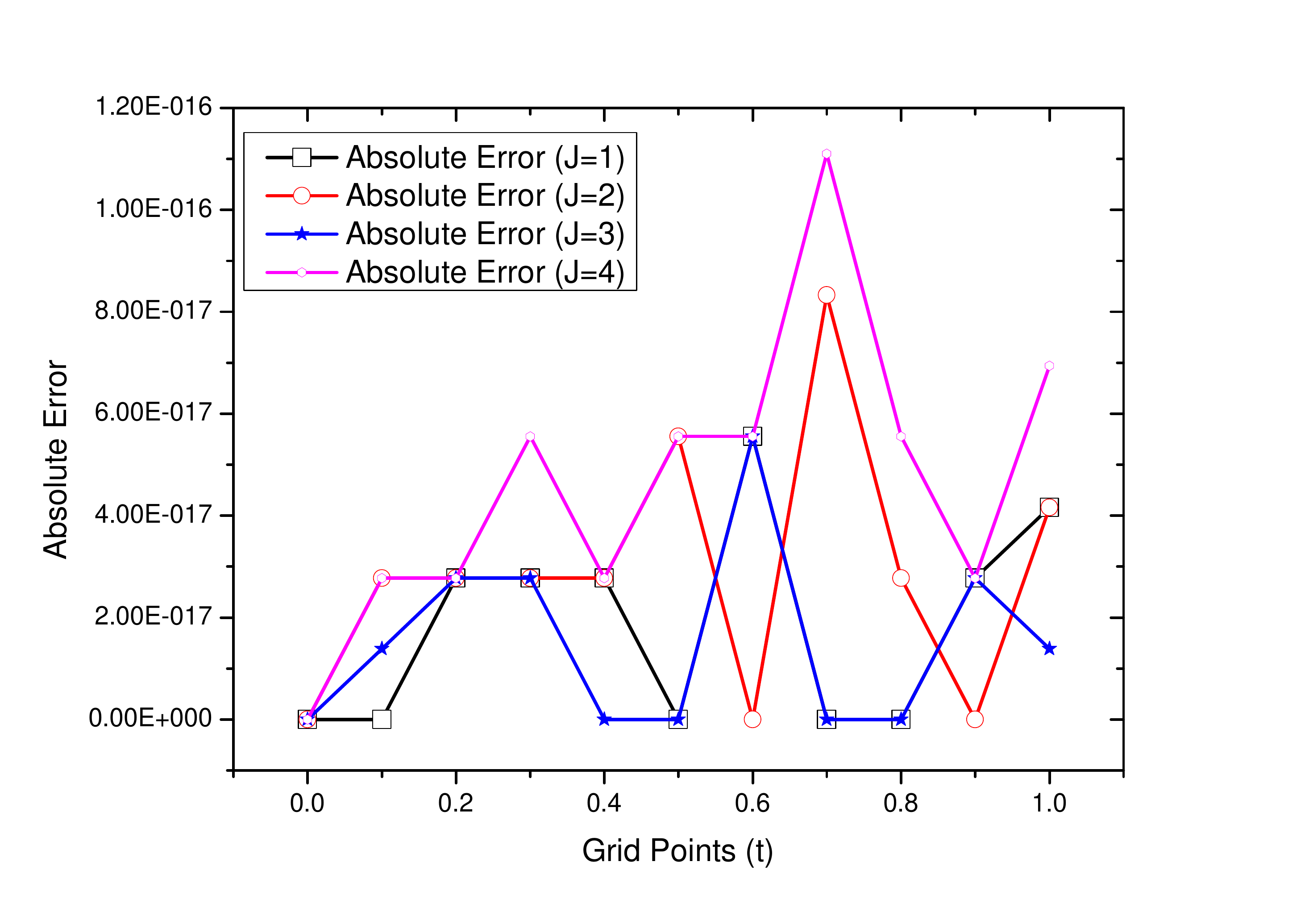}
\caption{Graph of absolute errors in $y^M(t)$  for example \ref{Problem5}.}\label{p2fig19}
        \end{subfigure}\hspace{0.5in}
        \begin{subfigure}[b]{0.45\textwidth}
\includegraphics[scale=0.3]{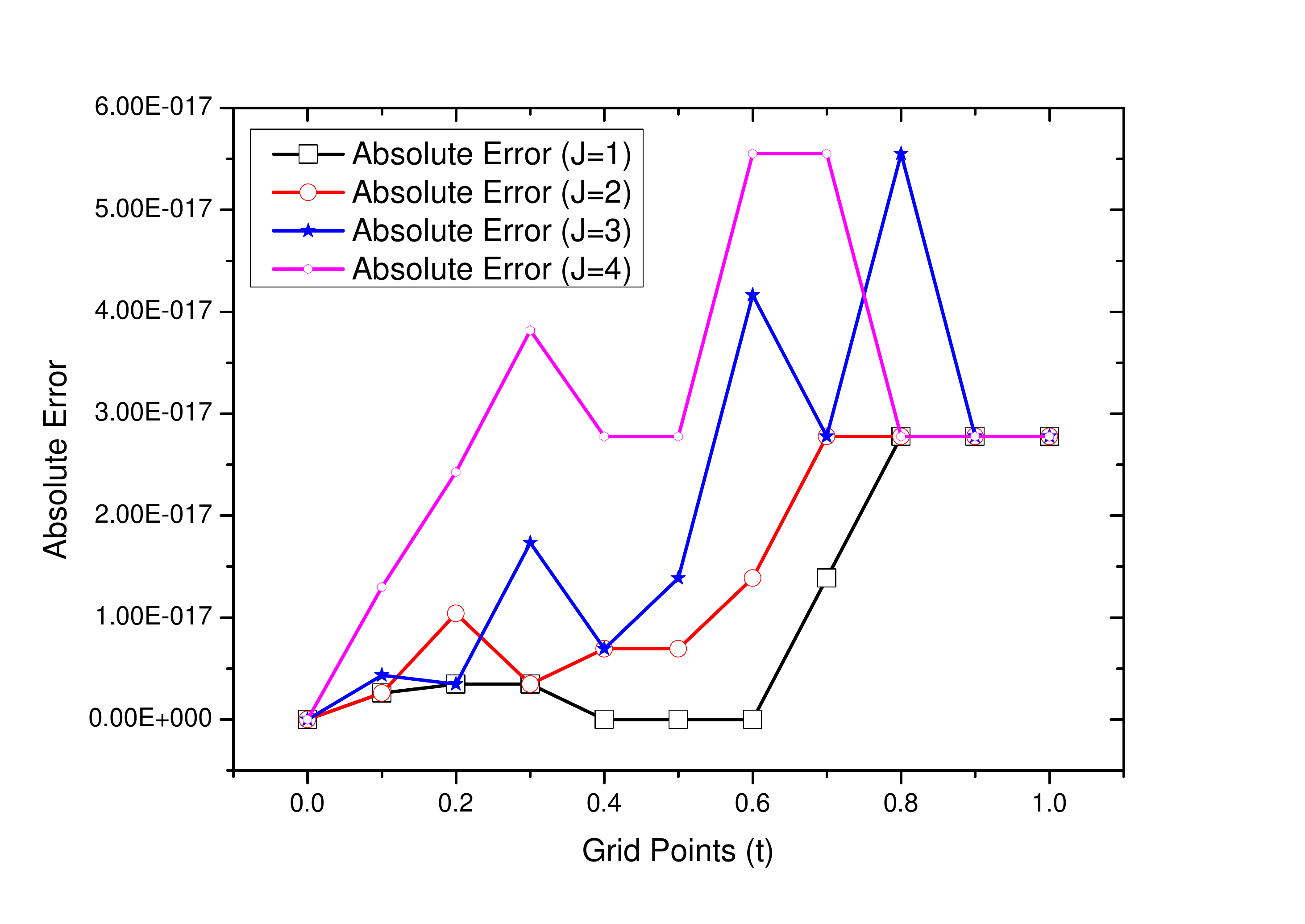}
\caption{Graph of absolute errors in $z^M(t)$  for example \ref{Problem5}.}\label{p2fig20}
        \end{subfigure}
        \caption{}
                \label{4pbvpfig2}
\end{figure}

\subsubsection{Example 6 (\cite{Barnwal2019})}  \label{Problem6}

Consider the system of singular differential equations \eqref{p2eqn1}-\eqref{p2eqn2} with four point boundary conditions \eqref{p2eq34}-\eqref{p2eq35}with $n_1=\frac{2}{3}$, $n_2=\frac{1}{3}$, $v_1=\frac{1}{2}$, $v_2=\frac{1}{4}$, $k_1=\frac{1}{2}$, $k_2=\frac{1}{2}$, $\omega_1=-\frac{1}{2}$, $\omega_2=-\frac{1}{2}$, $f_1(t,y(t),z(t))=\Big(-\frac{283}{216}+\frac{67}{9}t+\Big(\frac{80089}{16}-\frac{18961}{2}t+4489t^2\Big)\frac{t^2z^2}{729}-y^2z^2\Big)$ and  $f_2(t,y(t),z(t))=\Big(-\frac{1}{3}+\frac{3}{2}t+\Big(-67t+\frac{283}{4}\Big)^2 \Big(\frac{1}{4}t^2-\frac{2}{3}t+\frac{4}{9}\Big)\frac{t^4}{729}-y^2z^2\Big)$. Here we are computing absolute error described in subsubsection $\ref{Problem1}$. The exact solution $\tilde{y}(t)$ and $\tilde{z}(t)$ of this problem are $-\frac{67}{27}t^2+\frac{283}{108}t$ and $-\frac{1}{2}t^2+\frac{2}{3}t$, respectively.

For initial guess $[1.75,1.75,\cdots,1.75]$ computed solutions for $y^M(t)$ and $z^M(t)$ is given in Table \ref{p2Table11} and Table \ref{p2Table12}, respectively for $J=3$ and $J=4$. Graph for $y^M(t)$ and $z^M(t)$ with exact solution is given in figure \ref{p2fig21} and figure \ref{p2fig22}, respectively for $J=4$. Graph of absolute errors in computation of $y^M(t)$ and $z^M(t)$ is given in figure \ref{p2fig23} and figure \ref{p2fig24}, respectively for $J=1$, $J=2$, $J=3$ and $J=4$.

If the initial guesses are changed to $[1.79,1.79,\cdots,1.79]$ and $[1.71,1.71,\cdots,1.71]$ we observe that solution does not vary which proves that method is stable.

\begin{table}[H]											 
\centering											
\begin{center}											
\resizebox{9cm}{2.5cm}{											
\begin{tabular}	{|c | c|  c|  c|  c|}		
\hline
$t$	&	$J=3$	&	$J=4$	&	Exact Solution	&	Successive Iter. Tech. \cite{Barnwal2019}	\\\hline
0	&	0	&	0	&	0	&	0	\\\hline
0.1	&	0.237222	&	0.237222	&	0.237222	&	0.237222	\\\hline
0.2	&	0.424815	&	0.424815	&	0.424815	&	0.42481451	\\\hline
0.3	&	0.562778	&	0.562778	&	0.562778	&	0.5627774	\\\hline
0.4	&	0.651111	&	0.651111	&	0.651111	&	0.65111068	\\\hline
0.5	&	0.689815	&	0.689815	&	0.689815	&	0.68981435	\\\hline
0.6	&	0.678889	&	0.678889	&	0.678889	&	0.67888841	\\\hline
0.7	&	0.618333	&	0.618333	&	0.618333	&	0.61833284	\\\hline
0.8	&	0.508148	&	0.508148	&	0.508148	&	0.50814765	\\\hline
0.9	&	0.348333	&	0.348333	&	0.348333	&	0.34833284	\\\hline
1	&	0.138889	&	0.138889	&	0.138889	&	0.13888839	\\\hline
$L^\infty$	&	1.11022E-16	&	1.11022E-16	&		& 4.9806E-7		\\\hline

\end{tabular}}								
\end{center}
\caption{\small{Solutions $y^M(t)$ for $J=3$ and $J=4$  for example \ref{Problem6}.}}	
\label{p2Table11}	
										
\end{table}

\begin{table}[H]											 
\centering											
\begin{center}											
\resizebox{9cm}{2.5cm}{											
\begin{tabular}	{|c | c|  c|  c|  c|}		
\hline
$t$	&	$J=3$	&	$J=4$	&	Exact Solution	&	Successive Iter. Tech. \cite{Barnwal2019}	\\\hline
0	&	0	&	0	&	0	&	0	\\\hline
0.1	&	0.0616667	&	0.0616667	&	0.0616667	&	0.06166623	\\\hline
0.2	&	0.113333	&	0.113333	&	0.113333	&	0.11333273	\\\hline
0.3	&	0.155	&	0.155	&	0.155	&	0.15499929	\\\hline
0.4	&	0.186667	&	0.186667	&	0.186667	&	0.18666591	\\\hline
0.5	&	0.208333	&	0.208333	&	0.208333	&	0.20833259	\\\hline
0.6	&	0.22	&	0.22	&	0.22	&	0.21999934	\\\hline
0.7	&	0.221667	&	0.221667	&	0.221667	&	0.22166612	\\\hline
0.8	&	0.213333	&	0.213333	&	0.213333	&	0.21333293	\\\hline
0.9	&	0.195	&	0.195	&	0.195	&	0.19499974	\\\hline
1	&	0.166667	&	0.166667	&	0.166667	&	0.16666655	\\\hline
$L^\infty$	&	5.55112E-17	&	5.55112E-17	&		&	7.5635E-7	\\\hline
\end{tabular}}								
\end{center}
\caption{\small{Solutions $z^M(t)$ for $J=3$ and $J=4$  for example \ref{Problem6}.}}	
\label{p2Table12}											
\end{table}

\begin{figure}
        \begin{subfigure}[b]{0.45\textwidth}
\includegraphics[scale=0.3]{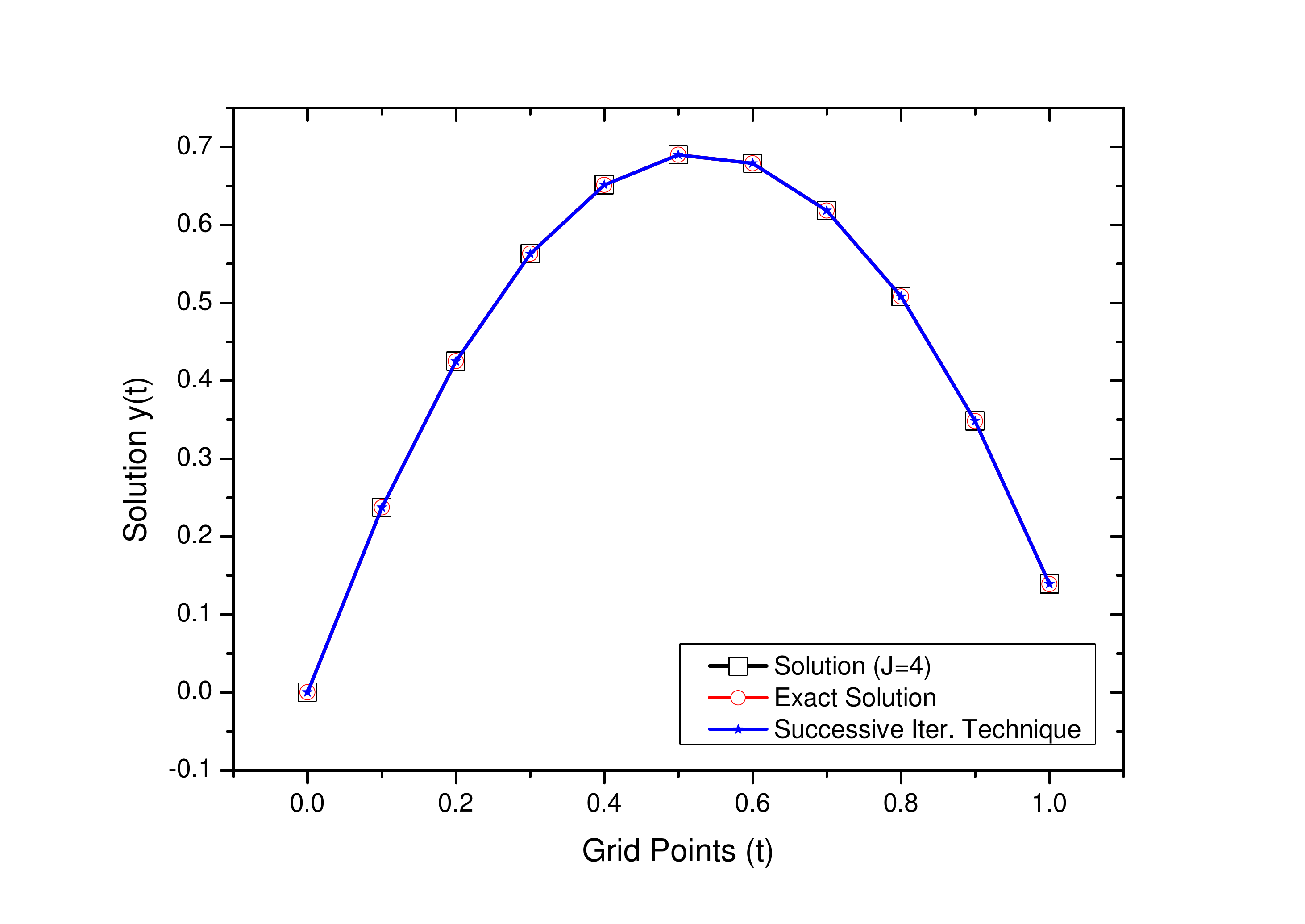}
\caption{Graph of $y^M(t)$ for $J=4$  for example \ref{Problem6}.}\label{p2fig21}
        \end{subfigure}\hspace{0.5in}
        \begin{subfigure}[b]{0.45\textwidth}
\includegraphics[scale=0.3]{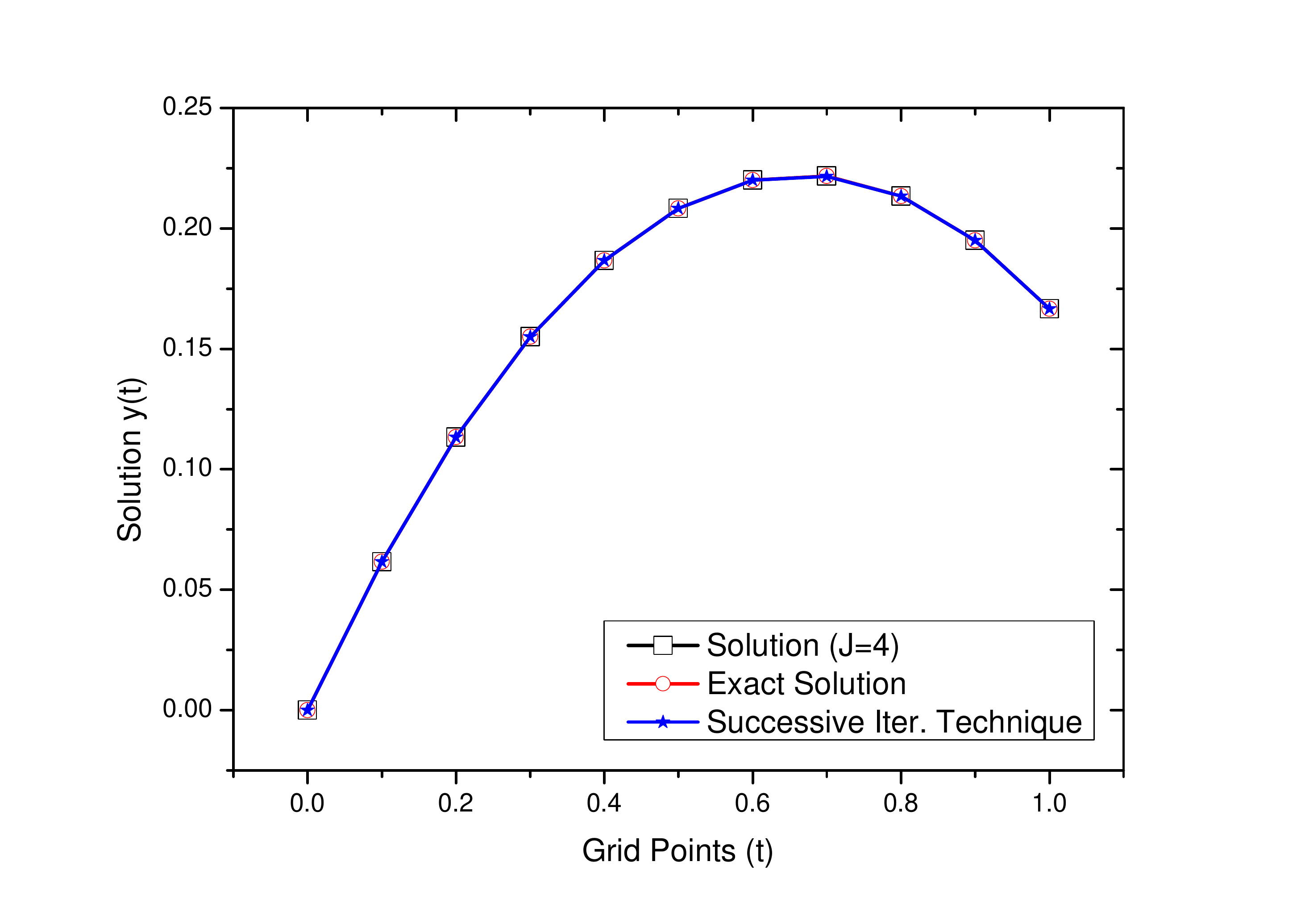}
\caption{Graph of $z^M(t)$ for $J=4$  for example \ref{Problem6}.}\label{p2fig22}
        \end{subfigure}
        \caption{}
                \label{4pbvpfig3}
\end{figure}

\begin{figure}
        \begin{subfigure}[b]{0.45\textwidth}
\includegraphics[scale=0.3]{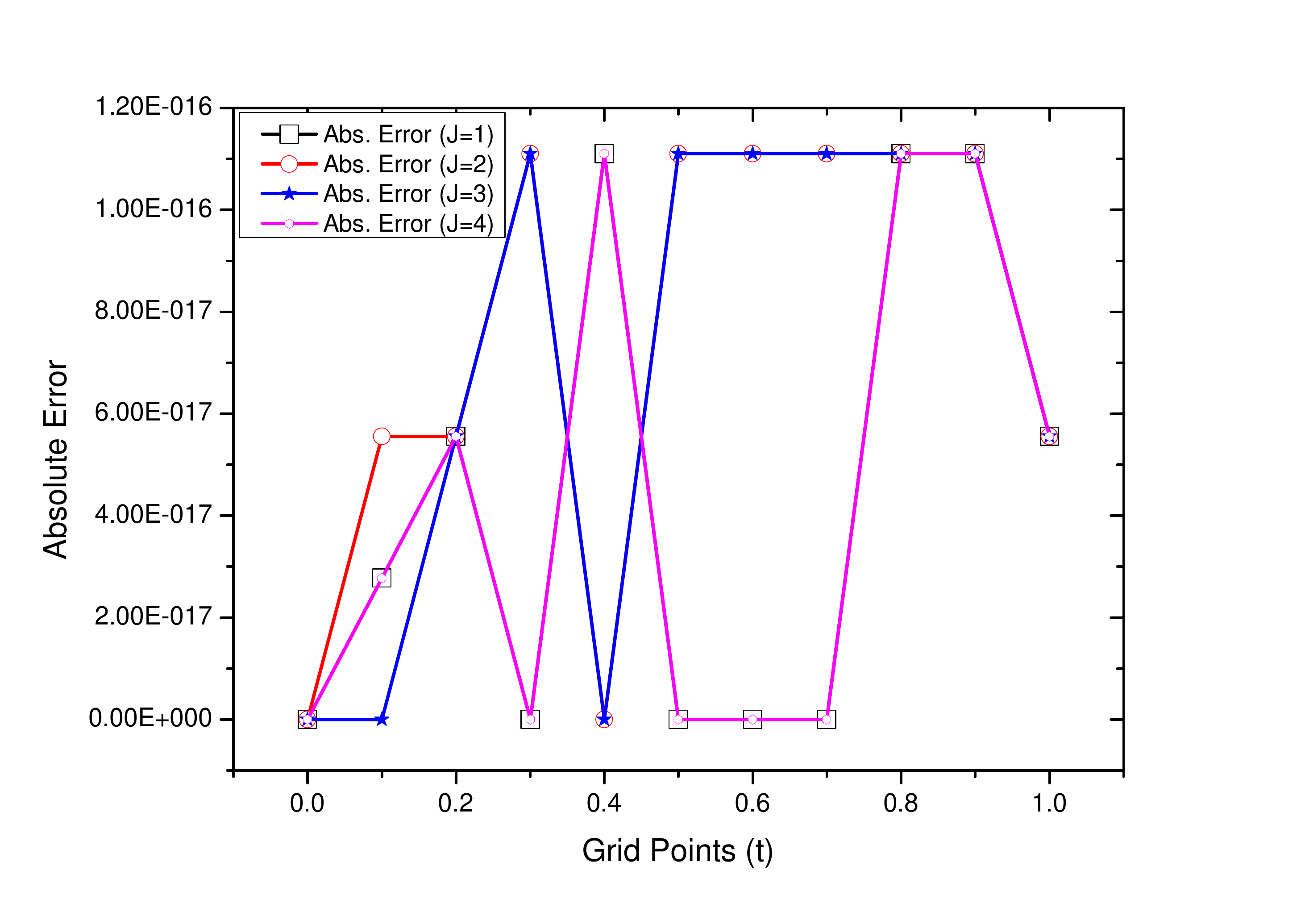}
\caption{Graph of absolute errors in $y^M(t)$  for example \ref{Problem6}.}\label{p2fig23}
        \end{subfigure}\hspace{0.5in}
        \begin{subfigure}[b]{0.45\textwidth}
\includegraphics[scale=0.3]{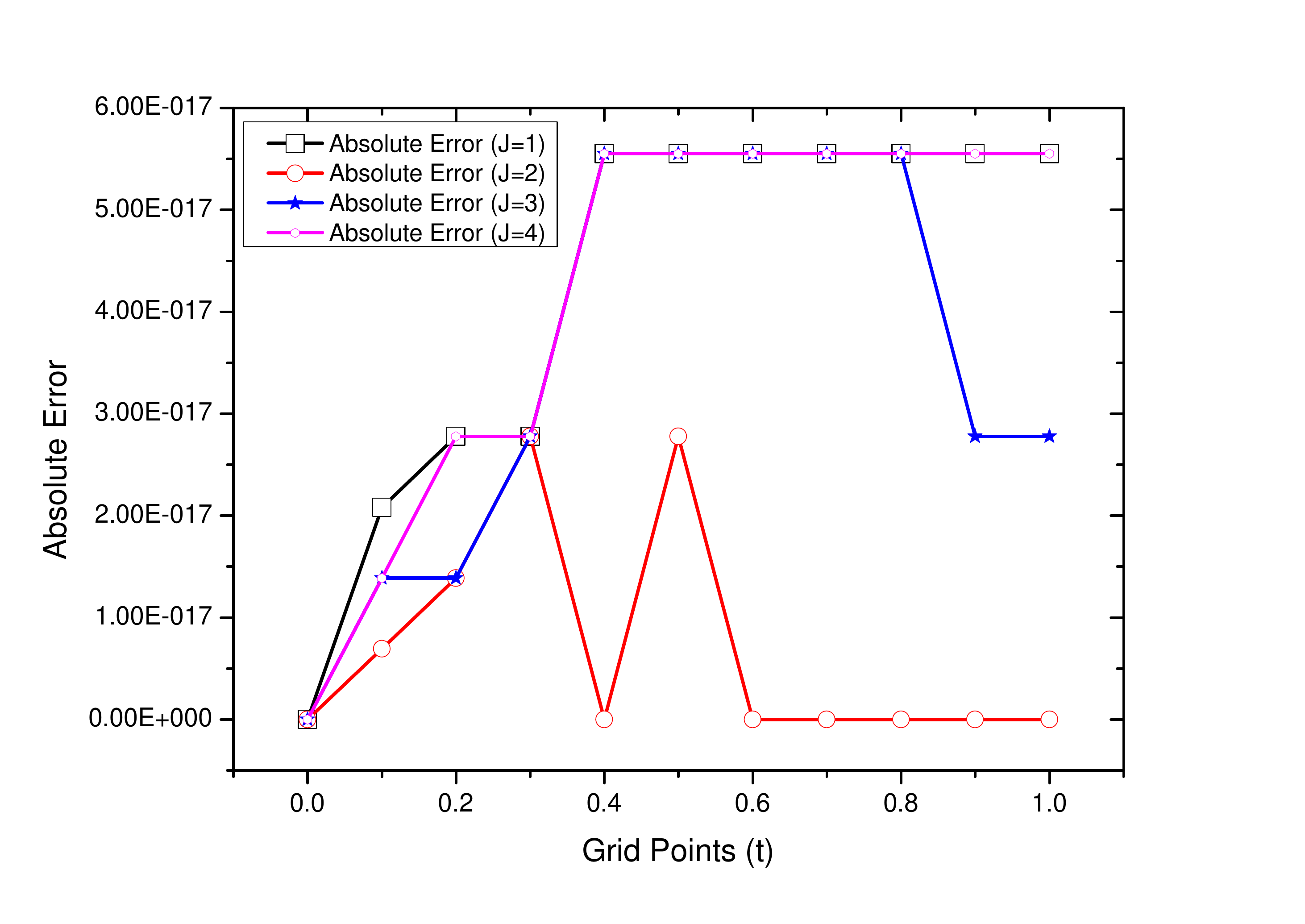}
\caption{Graph of absolute errors in $z^M(t)$  for example \ref{Problem6}.}\label{p2fig24}
        \end{subfigure}
        \caption{}
                \label{4pbvpfig4}
\end{figure}

\section{Remark} \label{remarkfinal}

In this work we have used Haar collocation method followed by Newton Raphson method to solve system of Lane-Emden equations subject to IVPs, BVPs and Four Point BVPs. 
For IVP we have given tables \ref{p2Table1}, \ref{p2Table2}, \ref{p2Table3}, \ref{p2Table4}, and figures \ref{ivpfig1}, \ref{ivpfig2}, \ref{ivpfig3}, \ref{ivpfig4}. It is visible from the table that at $J=3,4$, that is 16, 32 divisions 100\% accuracy is achieved.  Figures  \ref{ivpfig1}, \ref{ivpfig2}, \ref{ivpfig3}, \ref{ivpfig4} show the nature as well as accuracy of the computed solution.

For BVP we have given tables \ref{p2Table5}, \ref{p2Table6}, \ref{p2Table7}, \ref{p2Table8}, and figures \ref{bvpfig1}, \ref{bvpfig2}, \ref{bvpfig3}, \ref{bvpfig4}. It is visible from the table that at $J=3,4$, that is 16, 32 divisions absolute $L^\infty$ error is almost tending to zero. Figures  \ref{bvpfig1}, \ref{bvpfig2}, \ref{bvpfig3}, \ref{bvpfig4} show the nature as well as accuracy of the computed solution.

For the 4 point BVP we have given tables \ref{p2Table9}, \ref{p2Table10}, \ref{p2Table11}, \ref{p2Table12}, and figures \ref{4pbvpfig1}, \ref{4pbvpfig2}, \ref{4pbvpfig3}, \ref{4pbvpfig4}. It is visible from the table that at $J=3,4$, that is 16, 32 divisions absolute $L^\infty$ error is almost tending to zero. Figures  \ref{4pbvpfig1}, \ref{4pbvpfig2}, \ref{4pbvpfig3}, \ref{4pbvpfig4} show the nature as well as accuracy of the computed solution. We have compared our result with a recently published work by Barnwal et al. (Successive Iteration Technique \cite{Barnwal2019}). In our case error is almost zero while incase of  \cite{Barnwal2019} error is of order $10^{-6}$ or $10^{-7}$.

Our method takes minimum number of spatial division and only finite number of iterations (of Newton Raphson) to obtain exact solution. We observe that our solution approaches to exact solution as the value of $J$ is increased to $3,4$ or so. We also observed that if we makes small perturbation in initial guess then our solution does not vary significantly. Moreover our proposed method takes less time and computational work as compared to other numerical methods  like Monotonic iterative technique. The computed results shows the accuracy and stability of proposed method. Hence this method is more reliable and robust. 
\bibliographystyle{plain}
\bibliography{MasterNarendraPaper2}
\end{document}